\documentclass[10pt]{article}
\usepackage[utf8]{inputenc}
\usepackage[english]{babel}
\usepackage{xcolor}
\usepackage{algpseudocode}
\usepackage{algorithm}
\usepackage{array}

\usepackage{graphicx}
\usepackage{caption}
\usepackage{subcaption}

\usepackage{amsfonts}
\usepackage{amsthm,amsmath,amssymb}

\usepackage{authblk}

\usepackage{geometry}
\usepackage{amsthm}
\newtheorem{theorem}{Theorem}[section]
\newtheorem{lemma}[theorem]{Lemma}
\newtheorem{proposition}[theorem]{Proposition}
\newtheorem{problem}{Problem}[section]
\newtheorem{remark}[theorem]{Remark}
\newtheorem{definition}{Definition}[section]
\newtheorem{assumption}{Assumption}[section]

\usepackage{datetime}

\date{\displaydate{date}}

\numberwithin{equation}{section}

\usepackage{scalerel,stackengine}
\stackMath
\newcommand\reallywidehat[1]{%
\savestack{\tmpbox}{\stretchto{%
  \scaleto{%
    \scalerel*[\widthof{\ensuremath{#1}}]{\kern.1pt\mathchar"0362\kern.1pt}%
    {\rule{0ex}{\textheight}}%WIDTH-LIMITED CIRCUMFLEX
  }{\textheight}% 
}{2.4ex}}%
\stackon[-6.9pt]{#1}{\tmpbox}%
}
\begin{document}

\title{\LARGE \bf Identification of cavities and inclusions in linear elasticity with a phase-field approach}

\author[1]{Andrea Aspri}
\author[2]{Elena Beretta}
\author[3,5]{Cecilia Cavaterra}
\author[1,5]{Elisabetta Rocca}
\author[4]{Marco Verani}
\affil[1]{Department of Mathematics, Università degli Studi di Pavia}
\affil[2]{Department of Mathematics, NYU Abu Dhabi}
\affil[3]{Department of Mathematics, Università degli Studi di Milano}
\affil[4]{MOX, Department of Mathematics, Politecnico di Milano}
\affil[5]{IMATI-CNR Pavia}

\vspace{5mm}

\date{}

\maketitle

\thispagestyle{plain}
\pagestyle{plain}

\let\thefootnote\relax\footnotetext{


\thanks{}

}

\begin{abstract}
In this paper we deal with the inverse problem of determining  cavities and inclusions embedded in a linear elastic isotropic medium from boundary displacement's measurements. For, we consider a constrained minimization problem involving a boundary quadratic misfit functional with a regularization term that penalizes the perimeter of the cavity or inclusion to be identified.
Then using a phase field approach we derive a robust algorithm for the reconstruction of elastic inclusions and of cavities modelled as inclusions with a very small elasticity tensor. 
\end{abstract}

\vskip 10truemm

\section{Introduction}
The focus of this paper is the reconstruction of cavities and inclusions embedded in an elastic isotropic medium  by means of boundary tractions and displacements. 
%on the boundary of the domain. 
Identification of defects from boundary measurements plays an important role in non-destructive testing for damage assessment of mechanical specimens, which are possibly defective due to the presence of interior voids or cavities appearing during the manufacturing process, see, for instance, \cite{EilUrb18,KurMarIya17,NgoKasImbNguHui18,TroWelElv18} for possible applications to 3D-printing and additive manufacturing. {This kind of inverse problems has application also in medical imaging and in particular in elastography, a modality mapping the elastic properties and stiffness of soft tissue, \cite{Amm08,AmmBreEliGarJossKan15,Ammari,DouCara20,ShShQiZhCh21,ShKrHubDrSch20, WidSch15} (to cite a few), and in reflection seismology \cite{BrJaiKno05,Sym09}, a non invasive technique used by the oil and gas industry to map petroleum deposits in the Earth's upper crust and based on seismic data from land acquisition, see for example \cite{Shi}. We also mention some applications in volcanology, see for example \cite{Asp19,AspBerMas17,Seg10} and references therein.}

The underlying mathematical model is the following: Consider a bounded domain $\Omega\subset \mathbb{R}^d$, with $d=2,3$, representing the region occupied by an elastic isotropic medium and let $\partial\Omega=\Sigma_D \cup\Sigma_N$, with $\Sigma_D$ closed. Let the displacement field $u$ be solution to the following mixed boundary value problem for the Lam\'e system of linearized elasticity: 
%we take into account the following boundary value problem in the variable $u$ denoting the elastic 
%on the linear elasticity in an inhomogenenous isotropic medium. Let $\partial\Omega=\Sigma_D \cup\Sigma_N$, with $\Sigma_D$ closed, we take into account the following boundary value problem in the variable $u$ denoting the elastic displacement:
\begin{equation}\label{cavprob}
    \begin{cases}
    {\rm div}(\mathbb{C}_0\widehat\nabla u) = 0 &  \text{in } \Omega\setminus \overline{C},	\\
    (\mathbb{C}_0 \widehat\nabla u) n = 0 & \text{on}\ \partial C,\\
	(\mathbb{C}_0 \widehat\nabla u) \nu = g & \text{on } \Sigma_N,\\
		u=0 & \text{on } \Sigma_D,
    \end{cases}
\end{equation}
where $C\Subset\Omega$ is a cavity with Lipschitz boundary, and $\widehat{\nabla}u$ is the strain tensor.  $\mathbb{C}_0$ is a fourth-order isotropic elastic tensor, uniformly bounded and strongly convex, and $n$ and $\nu$ are the outer unit normal vector to $\partial C$ and $\Sigma_N$, respectively. The Neumann boundary datum $g$ is assumed  to be in $L^2(\Sigma_N)$. 

The forward problem consists in finding the elastic displacement $u$ in the elastic body occupying the region $\Omega$ induced by the tractions on $\Sigma_N$, given the cavity $C$. 
The inverse problem concerns the determination of the cavity $C$ from partial observations of $u$ on the boundary. More precisely, given measurements of the displacement, i.e. $u_{meas}\in L^2(\Sigma_N)$, find $C$ contained in $\Omega$, such that $u\lfloor_{\Sigma_N}=u_{meas}$, where $u\in H^1_{\Sigma_D}(\Omega\setminus C)$ is the solution to the forward problem. 

It is well known that this problem is severely ill-posed and only a very weak logarithmic conditional stability holds, assuming a-priori $C^{1,\alpha}$ regularity of the unknown cavities \cite{MR04}. A similar weak stability result holds also in the case of the determination of elastic inclusions, see for example \cite{MR16}. Hence, in general, the reconstruction of cavities and inclusions turns out to be a challenging issue.

To solve the problem we follow a similar strategy as in \cite{BRV2018,DEV2016} and  the one in \cite{BCP2021} for the reconstruction of conductivity inclusions and cavities respectively.   
Specifically, we consider the problem of minimizing the functional
\begin{equation}\label{defJ}
    J(C)=\frac12 \int_{\Sigma_N} |u(C)-u_{meas}|^2\, d\sigma(x)+\alpha \textrm{Per}(C),
\end{equation}
over a suitable set of cavities of finite perimeter and where $u(C)$ is the solution of (\ref{cavprob}) for a given cavity $C$, $\textrm{Per}(C)$ indicates the perimeter of $C$, and $\alpha$ is a positive regularization parameter.

We first investigate the continuity of solutions to (\ref{cavprob}) with respect to perturbations of the cavity $C$ in the Hausdorff distance topology and prove it using the Mosco convergence, see \cite{BB05,BHSZ01,G04}.
Similarly as in \cite{BCP2021}, continuity then allows us  to prove existence of minima of the functional $J(C)$, stability with respect to noisy data and convergence of the minimizers as $\alpha\to 0$ to the solution of the inverse problem. 

In the second part of the paper, we use a suitable phase-field relaxation of the functional $J$ in order to overcome issues arising from non-convexity and non-differentiability. 
To be more precise, we employ an idea adopted by Bourdin and Chambolle, \cite{BouCham03} in the context of topology optimization which consists in filling  the cavity with a fictitious elastic material described by an elastic tensor $\mathbb{C}_1:=\delta\mathbb{C}_0$, where $\delta$ is a small positive parameter and $\mathbb{C}_0$ has been extended to the whole domain $\Omega$. In this way, we transform the original inverse problem in the one of reconstructing  an elastic inclusion. {Then, since the identification of sharp interfaces is in general difficult to be treated numerically, we use a phase-field approach. Instead of binary (i.e., either 0 or 1) phase parameter $v$ describing  sharp interfaces between regions with two different materials we use a phase parameter $v$ as a $H^1$ scalar field, taking values in the interval $[0, 1]$. Then, we approximate the functional $J$ in \eqref{defJ} by means of a Ginzburg-Landau type functional (cf. \cite{Mod87})  
%As far as we know, the implementation of a phase-field approach is new in this context. The phase-field approach delivers an efficient method to solve the inverse boundary problem where boundary of the region filled by the material is unknown (cf. \cite{DEV2016}). Phase-field topology optimization penalizes an approximation of the interface perimeter in such a way that, by choosing a very small positive penalty term, one can approximate a sharp interface region separating areas with different elastic properties (cf.~\cite{BGHFS14}). The phase-field approach to structural optimization problems has been recently used by different authors (cf., e.g., \cite{ABCHDRR20,BGHFS14,CRBHRA19,GLNS21}), the main advantage being the fact that it allows to handle topology changes as well as nucleation of new holes. More in detail, a binary (i.e., either 0 or 1) phase parameter $v$ would call for sharp interfaces between regions with two different materials and since problems with  sharp interfaces are in general difficult to numerically treat and solve, as classically done, we consider the phase parameter $v$ as a $H^1$ scalar field, taking values in the interval $[0, 1]$. Then, we approximate the functional $J$ in \eqref{defJ} by means of the so-called Ginzburg-Landau type functional (cf. \cite{Mod87})
\begin{equation}
J_{\delta,\varepsilon}(v) := \frac12 \int_{\Sigma_N} |u_{\delta}(v)-u_{meas}|^2\, d\sigma(x)
+ \frac{4\alpha}{\pi} \int_{\Omega}\Big( \varepsilon|\nabla v|^2 + \frac{1}{\varepsilon}v(1-v)\Big)\, dx,\label{minrelI}
\end{equation}
where $\varepsilon$ is a small positive parameter, $\frac{4}{\pi}$ is a rescaled parameter in the Modica-Mortola relaxation of the perimeter, $u_\delta(v)$ denotes the solution of the modified boundary value problem: 
\begin{equation}
	\left\{
	\begin{aligned}
		\textrm{div}(\mathbb{C}_{\delta}(v) \widehat{\nabla} u_{\delta}(v)) &= 0 \qquad \text{in}\ \Omega,	\\
		(\mathbb{C}_{\delta}(v) \widehat\nabla u_{\delta}(v)) \nu &= g \qquad \text{on } \Sigma_N,\\
		u_{\delta}(v)&=0 \qquad \text{on } \Sigma_D,
	\end{aligned}
	\right.
	\label{elasticityI}
\end{equation}
where
\begin{equation}\label{eq:elasticity_tensor_inclusionI}
\mathbb{C}_{\delta}(v)= \mathbb{C}_0 +  (\mathbb{C}_1 - \mathbb{C}_0)v,\quad \textrm{with}\quad \mathbb{C}_1=\delta\mathbb{C}_0.
\end{equation}
Here $\mathbb{C}_0$ and $\mathbb{C}_1$ are the elasticity tensors in $\Omega \setminus C$ and $C$, respectively. 
Ideally, the optimal phase variable $v$ should be close to an ideal binary field. In fact, when $\varepsilon$ is small the potential term ($\int_{\Omega}\frac{1}{\varepsilon}v(1-v)\, dx$) prevails and the minimum is attained by a phase-field variable which takes mainly values close to $0$ and $1$ and the transition occurs in a thin layer of thickness of order $\varepsilon$.\ \\ 
The phase-field approach to structural optimization problems has been successfully used by different authors (cf., e.g., \cite{ABCHDRR20,BGHFS14,CRBHRA19,GLNS21}), the main advantage being the fact that it allows to handle topology changes as well as nucleation of new holes. }

{ To implement our algorithm in Section~\ref{secop} we provide first order necessary optimality conditions for the minimization problem associated to $J_{\delta,\varepsilon}$ whose discretized version is then employed in Section~\ref{sec:discretization} in order to develop the reconstruction algorithm. Minima of the functional $J_{\delta,\varepsilon}$ exist and the numerical experiments of Section 5 indicate that they are accurate approximations of minima of $J$, for $\varepsilon$ and $\delta$ sufficiently small. This fact could be rigorously justified proving that the $\Gamma$-convergence, as $\varepsilon$ and $\delta$ tend to $0$, to the functional $J$ holds, but this is still an open issue and will be the subject of a future research.}  
Some attempts along this direction have been done in the scalar case for example in \cite{BCP2021,RinRon11, Rondi11}.

The literature on reconstruction algorithms for identification of inclusions and cavities in elastostatic, viscoelastic and elastic waves systems is very rich and of big impact. In the case of small elastic inclusions or cavities, asymptotic expansions of the perturbed displacement have been used to detect position, size and shape from boundary measurements, see for example \cite{Kang03} and \cite{Ammari}.  The method followed in \cite{AmeBurHac07} is based on a shape derivative approach, both for elastic and thermoelastic problems. A topological gradient method has been applied in \cite{CarRap08}, for the detection of an elastic scatterer, and in \cite{CasFarGal12}, for identification of a cavity in time-harmonic wave elastic systems. Ikehata and Itou use the so-called  \textit{enclosure method} for the reconstruction of polygonal cavities in an elastostatic setting \cite{Ikehata11} and of a general cavity in a homogeneous isotropic viscoelastic body \cite{Ikehata12}. More recently, Doubova and Fern\'andez--Cara proposed an augmented Lagrangian method to identify rigid inclusions in a elastic waves system \cite{DouCara20}. Eberle and Harrach applied the \textit{monotonicity method} for the reconstruction of elastic inclusions using the monotonicity property of the Neumann-to-Dirichlet map \cite{EbeHar21}, and in  \cite{Lesnic} the authors used the method of fundamental solutions for the reconstruction of elastic cavities. For other reconstruction approaches we refer to the review paper \cite{BC} and references therein. 
Identification of cavities and elastic inclusions could be interpreted as a special case of the determination of Lam\'e parameters from boundary measurements, see for example \cite{AmmBreEliGarJossKan15,HubSheEkaNeuSch18} and \cite{Shi}.

The plan of the paper is the following. 
In Section \ref{sec:elastic problem} we investigate the continuity of the solution to the direct problem with respect to perturbations of the cavity in the Haussdorff topology and then derive the major properties of the misfit functional $J(C)$. In Section \ref{sec:filling the void} we consider the approximation of the cavity with an inclusion of small elasticity tensor, the corresponding misfit functional and its properties. We then introduce its phase-field relaxation and analyze its differentiability and derive necessary optimality conditions related to the phase-field minimization problem. In Section \ref{sec:discretization} we propose an iterative reconstruction algorithm allowing for the numerical approximation of the solution and prove its convergence properties. Finally, in Section 5 we present some numerical results showing the efficiency and robustness of the proposed reconstruction algorithm.

\subsection*{Notation and geometrical setting}
\label{Notation}
We introduce the principal notation used in the paper. 

\textit{Notation.} We denote scalar quantities, points, and vectors in italics, e.g.  ${x}, {y}$ and ${u}, {v}$, 
and fourth-order tensors in blackboard face, e.g.  $\mathbb{A}, \mathbb{B}$.

The symmetric part of a second-order tensor ${A}$ is denoted by
$\widehat{{A}}:=\tfrac{1}{2}\left({A}+{A}^T\right)$, where
${A}^T$ is the transpose matrix. In particular, $\widehat{\nabla} u$
represents the deformation tensor.
We utilize  standard notation for inner products, that is,
${u}\cdot {v}=\sum_{i} u_{i} v_{i}$, and  ${A}:
{B}=\sum_{i,j}a_{ij} b_{ij}$ ($B$ is a second-order tensor).
$|{A}|$ denotes the norm induced by the inner product on matrices:
	\begin{equation*}
		|{A}|=\sqrt{{A}:{A}}.
	\end{equation*}
\textit{Domains.} To represent locally a boundary as a graph of function, we adopt the notation:  $\forall\, x\in\mathbb{R}^d$, we set $x=(x',x_d)$, where $x'\in\mathbb{R}^{d-1}$, $x_d\in\mathbb{R}$, with $d=2,3$.
Given $r>0$, we denote by $B_{r}({x})\subset\mathbb{R}^d$ the set $B_{r}({x}):=\{(x',x_d)/\ |x'|^2+x_d^2<r^2\}$ and by $B'_{r}({x'})\subset\mathbb{R}^{d-1}$ the set $B'_{r}({x'}):=\{x'\in\mathbb{R}^{d-1}/\, |x'|^2<r^2\}$.
\begin{definition}[$C^{0,1}$ regularity]\label{def:lipschitz_domain}\ \\
		Let $\Omega$ be a bounded domain in $\mathbb{R}^d$. We say that a portion $\Sigma$ of $\partial \Omega$ is of Lipschitz class with constants $r_0$, $L_0$, if for any ${p}\in \Sigma$ there exists a rigid transformation of coordinates under which we have that ${p}$ is mapped to the origin  and
		\begin{equation*}
		\Omega\cap B_{r_0}({0})=\{{x}\in B_{r_0}({0})\, :\, x_d>\psi({x}')\},
		\end{equation*}
		where ${\psi}$ is a $C^{0,1}$ function on $B'_{r_0}({0})\subset \mathbb{R}^{d-1}$, such that
		\begin{equation*}
		\begin{aligned}
		{\psi}({0})&=0,\\
		\|{\psi}\|_{C^{0,1}(B'_{r_0}({0}))}&\leq L_0.
		\end{aligned}
		\end{equation*}
	\end{definition}
\noindent The Hausdorff distance between two sets $\Omega_1$ and $\Omega_2$ is defined by
    \begin{equation*}
        d_H(\Omega_1,\Omega_2)=\max\{\sup\limits_{x\in\Omega_1}\inf\limits_{y\in\Omega_2}\ dist(x,y),\sup\limits_{x\in\Omega_2}\inf\limits_{y\in\Omega_1}\ dist(x,y)\}.
    \end{equation*}
\textit{Functional setting:} Let $\Omega$ be a bounded domain. We set
\begin{equation}
BV(\Omega)= \{ v \in L^1(\Omega) \, : \, TV(v) < \infty \},
\label{BV}
\end{equation}
where
\begin{equation}\label{eq:TV}
TV(v) = \sup \left\{ \int_\Omega v \text{div}(\varphi); \quad \varphi \in C^1_0(\Omega), \,  \|{\varphi}\|_{{L^\infty}(\Omega)}\leq 1 \right\}
\end{equation}
is the total variation of $v$. The BV space is endowed with the natural norm $\|{v}\|_{BV(\Omega)} = \|{v}\|_{L^1(\Omega)} + TV(v)$. We recall that the perimeter of $\Omega$ is defined as
\begin{equation}\label{def:perimeter}
\textrm{Per}(\Omega)= TV(\chi_{\Omega}),
\end{equation}
where $\chi_{\Omega}$ is the characteristic function of the set $\Omega$. \\
Setting $H^1_{\partial\Omega}(\Omega):=\{\upsilon\in H^1(\Omega): \upsilon\lfloor_{\partial\Omega}=0\}$,
we recall the following inequalities.
\begin{proposition}\label{lemma1}
Let $\Omega$ be a bounded Lipschitz domain. For every $\upsilon\in H^1_{\partial\Omega}(\Omega)$, there exists a positive constant $\overline{c}=\overline{c}(\Omega)$ such that
\begin{equation}\label{korn}
(\textrm{Korn inequality})\qquad \|\nabla \upsilon \|_{L^2(\Omega)} \leq \overline{c}\ \|\widehat \nabla \upsilon\|_{L^2(\Omega)}.
\end{equation}
\begin{equation}\label{poincare}
(\textrm{Poincar\'e inequality})\qquad\|\upsilon\|_{H^1(\Omega)} \leq \overline{c}\  \|\nabla \upsilon\|_{L^2(\Omega)}.
\end{equation}
\end{proposition}
Estimates \eqref{korn} and \eqref{poincare} hold also in the case where $\upsilon$ is zero, in the trace sense, only on a portion of $\partial\Omega$.

\section{Elastic problem - detection of a cavity}\label{sec:elastic problem}
The focus of this work is the reconstruction of a cavity in an elastic body from boundary measurements using a phase-field approach. We assume that $\Omega$ is a bounded domain and that $\partial\Omega:=\Sigma_N\cup\Sigma_D$, with $|\Sigma_N| >0$, $|\Sigma_D|>0$, $\Sigma_D$ closed, where $\partial \Omega$ is of Lipschitz class with constants $r_0$ and $L_0$. Denoting by $C$ the cavity, we consider the mixed boundary value problem
\begin{equation}\label{prob:cavity}
    \begin{cases}
    {\rm div}(\mathbb{C}_0\widehat\nabla u) = 0 &  \text{in } \Omega\setminus \overline{C},	\\
    (\mathbb{C}_0 \widehat\nabla u) n = 0 & \text{on}\ \partial C,\\
	(\mathbb{C}_0 \widehat\nabla u) \nu = g & \text{on } \Sigma_N,\\
		u=0 & \text{on } \Sigma_D,
    \end{cases}
\end{equation}
where $n,\nu$ are the outer unit normal vector to $\partial C$ and $\partial \Omega$, respectively. \\
We make the following assumptions.
\begin{assumption}\label{ass:elasticity_tensor}
$\mathbb{C}_0=\mathbb{C}_0(x)$ is a fourth-order tensor such that  
	\begin{equation*}
	(\mathbb{C}_0)_{ijkh}(x)=(\mathbb{C}_0)_{jikh}(x)=(\mathbb{C}_0)_{khij}(x),\qquad  \forall 1\leq i,j,k,h\leq d,\ \textrm{and}\ \ {x}\in\Omega.
	\end{equation*}
	Moreover, $\mathbb{C}_0$ is assumed to be uniformly bounded and uniformly strongly convex, that is,
	$\mathbb{C}_0$ defines a positive-definite quadratic form on symmetric matrices:
	\begin{equation*}
	\mathbb{C}_0(x)\widehat{A}:\widehat{A}\geq \xi_0 
	|\widehat{A}|^{2}, \qquad \textrm{a.e in}\,\, \Omega,
	\end{equation*}
	for $\xi_0>0$.
\end{assumption}
\begin{remark}
We require that $\mathbb{C}_0$ is defined in $\Omega$, and not only in $\Omega\setminus C$, because we employ, in the second part of the paper, a reconstruction algorithm based on the strategy of filling the cavity with a fictitious elastic material.
\end{remark}
\begin{assumption}\label{ass:neumann_data} 
\begin{equation}
g\in L^2(\Sigma_N).    
\end{equation}
\end{assumption}
We assume Lipschitz regularity of the cavity (see Definition \ref{def:lipschitz_domain}), which is a typical requirement to prove uniqueness of the solution to the inverse problem,  see \cite{MR04}. More precisely, we make the following assumption. 
\begin{assumption}\label{ass:cavity}
Let 
\begin{center}
$C\in \mathcal{C}$:=\{$D \subset \overline{\Omega}:$ \hbox{ compact, simply connected} $\partial D\in C^{0,1}$ with constants $r_0$, $L_0$\, and ${dist}(D,\partial\Omega)\geq d_0>0$\}.
\end{center}
\end{assumption}
We define
\begin{equation}\label{eq:omega_d0}
\Omega^{d_0/2}=\{x\in\Omega\ / \ {dist}(x,\partial\Omega)\leq \frac{d_0}{2}\}. 
\end{equation}
For the class of admissible sets $\mathcal{C}$, the following result holds.
\begin{remark}\label{rem:compactness_sets}
$\mathcal{C}$ is compact with respect to the Hausdorff topology \cite{DalMaso93,MeRon13}.
\end{remark}
\begin{remark}
From now on, we will denote with $c$ any constant possibly depending on $\Omega$, $r_0$, $L_0$, $d$, $\xi_0$, $d_0$, $\overline{c}$, and on the uniform bounds of the elasticity tensor.
\end{remark}
Well-posedness of \eqref{prob:cavity} in $H^1_{\Sigma_D}(\Omega\setminus C)$  follows from an application of the Lax-Milgram theorem to the weak formulation of Problem \eqref{prob:cavity}:
\begin{center}
    \textrm{Find} $u\in H^1_{\Sigma_D}(\Omega\setminus C)$ \textrm{solution to}
    \begin{equation}\label{eq:weak_form_prob_cavity}
        \int_{\Omega\setminus C} \mathbb{C}_0\widehat{\nabla}{u}:\widehat{\nabla}{\varphi}\, dx=\int_{\Sigma_N}g\cdot \varphi\, d\sigma(x), \qquad \forall \varphi\in H^1_{\Sigma_D}(\Omega\setminus C),
    \end{equation}
\end{center}
(see for example \cite{Ciarlet88}).
Moreover, it holds
\begin{equation}\label{eq:H1_estimate}
    \|u\|_{H^1(\Omega\setminus C)}\leq c \|g\|_{L^2(\Sigma_N)}.
\end{equation}
Choosing $\varphi=u$ in \eqref{eq:weak_form_prob_cavity}, the last inequality follows from the strong convexity of the elasticity tensor $\mathbb{C}_0$ (see Assumption \ref{ass:elasticity_tensor}), from an application of the Korn and Poincar\'e inequality to the left-hand side of \eqref{eq:weak_form_prob_cavity} (see Proposition \ref{lemma1}), and from the use of a Cauchy-Schwarz inequality to the right-hand side. In fact,
\begin{equation}\label{eq:coercivity}
\int_{\Omega\setminus C} \mathbb{C}_0\widehat{\nabla}{u}:\widehat{\nabla}{u}\, dx\geq c \|\widehat{\nabla}u\|^2_{L^2(\Omega\setminus C)}\geq c \|\nabla u\|^2_{L^2(\Omega\setminus C)}\geq c \|u\|^2_{H^1(\Omega\setminus C)},
\end{equation}
and
\begin{equation}\label{eq:continuity_functional}
    \Bigg|\int_{\Sigma_N}g\cdot u\, d\sigma(x)\Bigg| \leq \|g\|_{L^2(\Sigma_N)}\|u\|_{L^2(\Sigma_N)}\leq c \|g\|_{L^2(\Sigma_N)}\|u\|_{H^1(\Omega\setminus C)},
\end{equation}
and so estimate \eqref{eq:H1_estimate} follows by \eqref{eq:coercivity} and \eqref{eq:continuity_functional}.
\\ 
Our aim is to tackle the following inverse problem: 
\begin{problem} \label{IP_cavity}
Under Assumptions \ref{ass:elasticity_tensor},  \ref{ass:neumann_data}, and \ref{ass:cavity}, given $\displaystyle u_{meas} \in L^2(\Sigma_N)$, find  $C \in \mathcal{C}$ such that $u\lfloor_{\Sigma_N} = u_{meas}$, where $u\in H^1_{\Sigma_D}(\Omega\setminus C)$ solves \eqref{prob:cavity}.
\end{problem}
It has been proved in \cite{MR04} (see also \cite{ABR18}) that Problem \ref{IP_cavity} has a unique solution when $\partial C$ is of Lipschitz class. Logarithmic stability estimates have been proved under the assumption of $C^{1,\sigma}$ regularity, $0<\sigma\leq 1$, on the cavity $C$, cf. \cite{MR04}.   

For the reconstruction of the solution to the inverse problem we consider a standard approach based on the minimization of a quadratic misfit functional, with a Tikhonov regularization penalizing the perimeter of $C$. More precisely, let
\begin{equation}\label{eq:JC}
\min_{C\in \mathcal{C}}J(C), \hbox{ where } J(C)=\frac12 \int_{\Sigma_N} |u(C)-u_{meas}|^2\, d\sigma(x)+\alpha \textrm{Per}(C),
\end{equation}
where $\alpha>0$ represents a regularization parameter,  $\textrm{Per}(C)$ the perimeter of the set $C$, see \eqref{def:perimeter}, and $u(C)\in H^1_{\Sigma_D}(\Omega\setminus C)$ the solution to \eqref{eq:weak_form_prob_cavity}.

\subsection{Continuity property of solutions with respect to $C$}
Adapting to our case some known results in literature, see for example \cite{ChDo97,BV00,BB05,G04,LRX19} and references therein, in this section we will show the continuity of the boundary term in \eqref{eq:JC} with respect to perturbations of the cavity $C$ in the Hausdorff distance.
%Among the vast literature on continuity results for Dirichlet and Neumann problems, we refer the reader to \cite{ChDo97,BV00,BB05,G04,LRX19} and references therein. 

To this purpose, we recall the definition of Mosco convergence and some of its properties (see \cite{BHSZ01,BB05,G04,MeRon13}). 
Let $X$ be a reflexive Banach space, and $G_n$ a sequence of closed subspaces of $X$. We define
\begin{equation}
    G':=\{x\in X\ /\  x=w-\limsup y_{n_k}\ ,\ y_{n_k}\in G_{n_k}, \ n_k\to +\infty     \}
\end{equation}
and
\begin{equation}
    G'':=\{x\in X\ /\ x=s-\liminf y_{n}\ ,\ y_{n}\in G_{n}\ \textrm{for}\ n\ \textrm{large}\}.
\end{equation}
$G',G''$ are called the \textit{weak-limsup} and \textit{the strong-liminf} of the sequence $G_n$ in the sense of Mosco.
\begin{definition}
The sequence $G_n$ converges in the sense of Mosco if $G'=G''=G$. $G$ is called the Mosco limit of $G_n$.
\end{definition}
In other words, $G_n$ converges in the sense of Mosco to $G$ when the following two conditions hold:
\begin{align}
&\textrm{If}\ u_{n_k}\in G_{n_k}\ \textrm{is such that}\ u_{n_k}\rightharpoonup u\ \textrm{in}\ X, \textrm{then}\ u\in G;\label{eq:second_condition}\\
&\forall u\in G, \exists u_n\in G_n\ \textrm{such that}\ u_n\to u\ \textrm{in}\ X.\label{eq:first_condition}
\end{align}
Given $\Omega$ and $\Omega\setminus C$, we can identify the Sobolev space $H^1_{\Sigma_D}(\Omega\setminus C)$ with a closed subspace of $L^2(\Omega,\mathbb{R}^{d+d^2})$ through the map
    \begin{equation}\label{eq:identification}
    \begin{aligned}
        H^{1}_{\Sigma_D}(\Omega\setminus C) &\hookrightarrow L^2(\Omega,\mathbb{R}^{d+d^2})\\
       u &\to (u,\partial_{j}u_i),\qquad \forall i,j=1,\cdots,d
    \end{aligned}
    \end{equation}
with the convention of extending $u$ and $\nabla u$ to zero in $C$. The same identification holds for $\Omega\setminus C_n$, extending $u_n$ and $\nabla u_n$ to zero in $C_n$. \\
Since we are considering the case of uniform Lipschitz domains, we have the following result, which is an adaptation of Theorem 7.2.7 in \cite{BB05}. 
\begin{theorem}\label{th:Mosco_conv}
Let us assume that $C_n, C\subset \Omega$ belong to the class $\mathcal{C}$. If $C_n\to C$ in the Hausdorff metric, then $H^1_{\Sigma_D}(\Omega\setminus C_n)$ converges to $H^1_{\Sigma_D}(\Omega\setminus C)$ in the sense of Mosco. 
\end{theorem}
We can now prove the following continuity result.
\begin{theorem}\label{th:continuity}
Let $C_n\in\mathcal{C}$ be a sequence of sets converging to $C$ in the Hausdorff metric (cf. Remark \ref{rem:compactness_sets}), and let $u(C_n)=:u_n\in H^1_{\Sigma_D}(\Omega\setminus C_n)$, $u(C)=:u\in H^1_{\Sigma_D}(\Omega\setminus C)$ be solutions of \eqref{eq:weak_form_prob_cavity} in $\Omega\setminus C_n$, $\Omega\setminus C$, respectively. Then
\begin{equation}\label{eq:continuity}
    \lim\limits_{n\to +\infty} \int_{\Sigma_N}|u_n-u|^2\,d\sigma(x)=0.
\end{equation}
\end{theorem}
\begin{proof}
Thanks to the uniform Lipschitz regularity of $\partial(\Omega\setminus C_n)$ (and $\partial(\Omega\setminus C)$), we have that the Korn and Poincar\'e inequalities are uniform with respect to $n$ in $H^1_{\Sigma_D}(\Omega\setminus C_n)$, since they depend only on the Lipschitz constants of the domain $\partial(\Omega\setminus C_n)$, see \cite{AMR2008,Ch75}. 
Therefore, from \eqref{eq:weak_form_prob_cavity} and \eqref{eq:H1_estimate}, we have that 
\begin{equation}\label{eq:uniform_estimate_H1}
    \|u_n\|_{H^1(\Omega\setminus C_n)}\leq c,
\end{equation}
where $c$ is independent of $n$. \\
Hence, from the identification \eqref{eq:identification}, we get that $\|u_n\|_{L^2(\Omega,\mathbb{R}^{d+d^2})}$ is uniformly bounded.
%\begin{equation*}
%    (u_n,\nabla u_n)\ \textrm{is bounded in}\ L^2(\Omega,\mathbb{R}^{d+d^2}). 
%\end{equation*}
Up to subsequences, there exists $u^*\in L^2(\Omega,\mathbb{R}^{d+d^2})$ such that
\begin{equation*}
%    (u_n,\nabla u_n)
    u_n \rightharpoonup u^*\ \textrm{in}\ L^2(\Omega,\mathbb{R}^{d+d^2}). 
\end{equation*}
Thanks to Theorem \ref{th:Mosco_conv} and from the first condition of the Mosco convergence applied to $G_n=H^1_{\Sigma_D}(\Omega\setminus C_n)$, $G=H^1_{\Sigma_D}(\Omega\setminus C)$, and $X=L^2(\Omega,\mathbb{R}^{d+d^2})$, see \eqref{eq:second_condition}, we have that $u^*\in H^1_{\Sigma_D}(\Omega\setminus C)$.\\
Moreover, taking $\varphi\in H^1_{\Sigma_D}(\Omega\setminus C)$, there exists $\varphi_n\in H^1_{\Sigma_D}(\Omega\setminus C_n)$ by \eqref{eq:first_condition} such that 
\begin{equation}\label{eq:strong_convergence_phi}
\varphi_n\to \varphi\ \textrm{in}\ L^2(\Omega,\mathbb{R}^{d+d^2}). 
%(\varphi_n,\nabla \varphi_n)\to (\varphi,\nabla\varphi)\ \textrm{in}\ L^2(\Omega,\mathbb{R}^{d+d^2}). 
\end{equation}
Considering the weak formulation for $u_n$ (see \eqref{eq:weak_form_prob_cavity} specialized to the case with $C=C_n$ and $\varphi=\varphi_n$)
\begin{equation}\label{eq:weak_form_Mosco}
\int_{\Omega\setminus C_n} \mathbb{C}_0\widehat{\nabla}u_n:\widehat{\nabla}\varphi_n\, dx=\int_{\Sigma_N}g\cdot \varphi_n\, d\sigma(x),
\end{equation}
and since $\varphi_n\in H^1_{\Sigma_D}(\Omega\setminus C_n)$ and $\varphi\in H^1_{\Sigma_D}(\Omega\setminus C)$, it holds
\begin{equation*}
\int_{\Sigma_N}g\cdot \varphi_n\, d\sigma(x)=\int_{\Sigma_N}g\cdot(\varphi_n-\varphi)\, d\sigma(x)+\int_{\Sigma_N}g\cdot\varphi\, d\sigma(x).
\end{equation*}
Hence, thanks to Assumption \ref{ass:cavity} and \eqref{eq:strong_convergence_phi}, we have
\begin{equation*}
\begin{aligned}
 \Bigg|\int_{\Sigma_N}g\cdot(\varphi_n-\varphi)\, d\sigma(x)\Bigg|&\leq c \|g\|_{L^2(\Sigma_N)} \|\varphi_n-\varphi\|_{L^2(\Sigma_N)}\\
&\leq c \|\varphi_n-\varphi\|_{H^1_{\Sigma_D}(\Omega^{d_0/2})} \to 0,  
\end{aligned}
\end{equation*}
as $n\to +\infty$, where $\Omega^{d_0/2}$ is defined as in \eqref{eq:omega_d0}. Therefore,
\begin{equation}
    \int_{\Sigma_N}g\cdot \varphi_n\, d\sigma(x) \to \int_{\Sigma_N}g\cdot \varphi\, d\sigma(x),\qquad \textrm{as}\ n\to+\infty.
\end{equation}
The term on the left-hand side of \eqref{eq:weak_form_Mosco} is equal to
\begin{equation}\label{eq:conv}
\int_{\Omega\setminus C_n} \mathbb{C}_0\widehat{\nabla}u_n:\widehat{\nabla}\varphi_n\, dx=\int_{\Omega\setminus C_n} \mathbb{C}_0\widehat{\nabla}u_n:\widehat{\nabla}(\varphi_n-\varphi)\, dx+\int_{\Omega\setminus C_n} \mathbb{C}_0\widehat{\nabla}u_n:\widehat{\nabla}\varphi\, dx.
\end{equation}
Then, by \eqref{eq:uniform_estimate_H1} and \eqref{eq:strong_convergence_phi}, it follows 
\begin{equation}\label{eq:mosco2}
\Bigg|\int_{\Omega\setminus C_n} \mathbb{C}_0\widehat{\nabla}u_n:\widehat{\nabla}(\varphi_n-\varphi)\, dx\Bigg|\leq c \|\widehat{\nabla}u_n\|_{L^2(\Omega\setminus C_n)} \|\widehat{\nabla}(\varphi_n-\varphi)\|_{L^2(\Omega\setminus C_n)} \to 0,  
\end{equation}
as $n\to +\infty$. Analogously, for the second integral on the right-hand side of \eqref{eq:conv}, using the symmetries of the elasticity tensor, we get
\begin{equation}\label{eq:mosco3}
\begin{aligned}
    \int_{\Omega\setminus C_n} \mathbb{C}_0\widehat{\nabla}u_n:\widehat{\nabla}\varphi\, dx&=  \int_{\Omega\setminus C_n} \widehat{\nabla}u_n:\mathbb{C}_0\widehat{\nabla}\varphi\, dx\\
    &\to \int_{\Omega\setminus C} \widehat{\nabla}u^*:\mathbb{C}_0\widehat{\nabla}\varphi\, dx= \int_{\Omega\setminus C} \mathbb{C}_0\widehat{\nabla}u^*:\widehat{\nabla}\varphi\, dx,
\end{aligned}    
\end{equation}
as $n\to +\infty$. Consequently, using \eqref{eq:mosco2} and \eqref{eq:mosco3} in \eqref{eq:conv}, we get   
\begin{equation}
    \int_{\Omega\setminus C_n} \mathbb{C}_0\widehat{\nabla}u_n:\widehat{\nabla}\varphi_n\, dx \to \int_{\Omega\setminus C} \mathbb{C}_0\widehat{\nabla}u^*:\widehat{\nabla}\varphi\, dx, \qquad \textrm{as}\ n\to+\infty.
\end{equation}
Therefore, we find that 
\begin{equation*}
\begin{aligned}
    \int_{\Omega\setminus C}\mathbb{C}_0\widehat{\nabla}u^*:\widehat{\nabla}\varphi\, dx&=\int_{\Sigma_N}g\cdot \varphi\, d\sigma(x)\\
    &=\int_{\Omega\setminus C}\mathbb{C}_0\widehat{\nabla}u:\widehat{\nabla}\varphi\, dx, \qquad \forall\varphi\in H^1_{\Sigma_D}(\Omega\setminus C),
\end{aligned}    
\end{equation*}
where the last equality comes from the weak formulation \eqref{eq:weak_form_prob_cavity}. Therefore,
\begin{equation*}
    \int_{\Omega\setminus C}\mathbb{C}_0\widehat{\nabla}(u^*-u):\widehat{\nabla}\varphi\, dx=0, \qquad \forall\varphi\in H^1_{\Sigma_D}(\Omega\setminus C),
\end{equation*}
so that $u^*=u$. This conclusion comes from the choice $\varphi=u^*-u$, and the use of Assumption \ref{ass:elasticity_tensor} and Korn and Poincar\`e inequalities (see Proposition \ref{lemma1}).

Next, we prove that $u_n\to u$ in $L^2(\Sigma_N)$ by showing strong convergence of $u_n$ to $u$ in $H^1$-norm in a neighborhood of the boundary of $\Omega$. Consider the weak formulations
\begin{equation}\label{eq:weak_form_un}
    \int_{\Omega\setminus C_n}\mathbb{C}_0 \widehat{\nabla}u_n : \widehat{\nabla}\varphi_1\, dx= \int_{\Sigma_N}g\cdot \varphi_1\, d\sigma(x),\qquad  \forall \varphi_1\in H^1(\Omega\setminus C_n),
\end{equation}
\begin{equation}\label{eq:weak_form_u}
    \int_{\Omega\setminus C}\mathbb{C}_0 \widehat{\nabla}u : \widehat{\nabla}\varphi_2\, dx= \int_{\Sigma_N}g\cdot \varphi_2\, d\sigma(x),\qquad  \forall \varphi_2\in H^1(\Omega\setminus C).
\end{equation}
Now, we define $\Phi=(u_n-u)\chi^2$, where $\chi$ is a smooth cut-off function, $\chi\in[0,1]$ in $\Omega$, such that
\begin{equation*}
\chi=
    \begin{cases}
    1 & \textrm{in}\ \overline{\Omega}^{d_0/4}\\
    0 & \textrm{in}\ \Omega\setminus \Omega^{d_0/2}.
    \end{cases}
\end{equation*}
Then, we choose $\varphi_1=\varphi_2=\Phi$ in \eqref{eq:weak_form_un} and \eqref{eq:weak_form_u}, that is
\begin{equation*}
    \int_{\Omega^{d_0/2}}\mathbb{C}_0 \widehat{\nabla}u_n : \widehat{\nabla}\left((u_n-u)\chi^2\right)\, dx= \int_{\Sigma_N}g\cdot (u_n-u)\, d\sigma(x),
\end{equation*}
\begin{equation*}
    \int_{\Omega^{d_0/2}}\mathbb{C}_0 \widehat{\nabla}u : \widehat{\nabla}\left((u_n-u)\chi^2\right)\, dx= \int_{\Sigma_N}g\cdot (u_n-u)\, d\sigma(x).
\end{equation*}
Subtracting the last two equations, we find
\begin{equation*}
\int_{\Omega^{d_0/2}}\mathbb{C}_0 \widehat{\nabla}(u_n-u) : \widehat{\nabla}\left((u_n-u)\chi^2\right)\, dx=0,
\end{equation*}
that is,
\begin{equation}\label{eq:diff_u_n-u}
\begin{aligned}
&\int_{\Omega^{d_0/2}}\chi^2\mathbb{C}_0 \widehat{\nabla}(u_n-u) : \widehat{\nabla}(u_n-u)\, dx\\
&\hspace{4cm}+
\int_{\Omega^{d_0/2}}2\chi \mathbb{C}_0 \widehat{\nabla}(u_n-u) : \reallywidehat{\left((u_n-u)\otimes \nabla \chi\right)}\, dx=0.
\end{aligned}
\end{equation}
On the second integral, we apply the Young's inequality with a suitable parameter $\kappa>0$, that is
\begin{equation*}
\begin{aligned}
    &\int_{\Omega^{d_0/2}}2\chi \mathbb{C}_0 \widehat{\nabla}(u_n-u) : \reallywidehat{\left((u_n-u)\otimes \nabla \chi\right)}\, dx\\
    &\hspace{4cm}\leq 4\kappa\int_{\Omega^{d_0/2}}\chi^2 \mathbb{C}_0\widehat{\nabla}{(u_n-u)}:\widehat{\nabla}{(u_n-u)}\\
    &\hspace{4cm}+\frac{1}{\kappa}\int_{\Omega^{d_0/2}}\mathbb{C}_0\reallywidehat{\left((u_n-u)\otimes \nabla \chi\right)}:\reallywidehat{\left((u_n-u)\otimes \nabla \chi\right)}.
\end{aligned}    
\end{equation*}
Hence, using this last inequality in \eqref{eq:diff_u_n-u}, we get
\begin{equation*}
\begin{aligned}
    &(1-4\kappa) \int_{\Omega^{d_0/2}}\chi^2 \mathbb{C}_0 \widehat{\nabla}(u_n-u) : \widehat{\nabla}(u_n-u)\, dx\\
    &\hspace{4cm}\leq \frac{1}{\kappa} \int_{\Omega^{d_0/2}}\mathbb{C}_0\reallywidehat{\left((u_n-u)\otimes \nabla \chi\right)}:\reallywidehat{\left((u_n-u)\otimes \nabla \chi\right)}.
\end{aligned}    
\end{equation*}
The right-hand side integral goes to zero, noticing that 
\begin{equation}\label{eq:estimate1}
    \begin{aligned}
    &\int_{\Omega^{d_0/2}}\mathbb{C}_0\reallywidehat{\left((u_n-u)\otimes \nabla \chi\right)}:\reallywidehat{\left((u_n-u)\otimes \nabla \chi\right)}\\
    &\leq c \int_{\Omega^{d_0/2}} \mathbb{C}_0 |u_n-u|^2 |\nabla \chi|^2\, dx\leq c \int_{\Omega^{d_0/2}} |u_n-u|^2\, dx \longrightarrow 0,\quad \textrm{as}\, n\to+\infty.
    \end{aligned}
\end{equation}
The left-hand side can be estimated using the fact that 
\begin{equation*}
    \int_{\Omega^{d_0/2}}\chi^2 \mathbb{C}_0 \widehat{\nabla}(u_n-u) : \widehat{\nabla}(u_n-u)\, dx\geq \int_{\Omega^{d_0/4}}\mathbb{C}_0 \widehat{\nabla}(u_n-u) : \widehat{\nabla}(u_n-u)\, dx
\end{equation*}
and, then, by means of the Korn inequality 
\begin{equation}\label{eq:estimate2}
    \int_{\Omega^{d_0/4}}\mathbb{C}_0 \widehat{\nabla}(u_n-u) : \widehat{\nabla}(u_n-u)\, dx\geq c \|\nabla(u_n-u)\|^2_{L^2(\Omega^{d_0/4})}.
\end{equation}
From \eqref{eq:estimate2} and \eqref{eq:estimate1}, and recalling that $u_n$ is converging strongly in $L^2$-norm to $u$ from the previous results, we find that
\begin{equation}
    \|u_n-u\|_{H^1(\Omega^{d_0/4})}\to 0,\quad \textrm{as}\ n\to +\infty.
\end{equation}
Finally, by the continuity of the trace theorem the proof is concluded.
\end{proof}
\begin{remark}
In the previous result, $u_n\to u$ in $L^2(\Sigma_N)$ can be also proved using the following arguments: note that the trace operator is a linear continuous operator from $H^1_{\Sigma_D}(\Omega\setminus C_n)$ to $H^{\frac{1}{2}}(\Sigma_N)$ (and, analogously, from $H^1_{\Sigma_D}(\Omega\setminus C)$ to $H^{\frac{1}{2}}(\Sigma_N)$), hence is also continuous in the weak topology, see \cite{Brezis83}. Moreover, since $H^{\frac{1}{2}}(\Sigma_N) \hookrightarrow L^2(\Sigma_N)$ is compact, we find that $u_n\to u$ in $L^2(\Sigma_N)$.
\end{remark}
As a consequence of the continuity of the boundary functional, some properties of the functional $J(C)$ defined in \eqref{eq:JC} follow.

\begin{proposition}
 For every $\alpha >0$ there exists at least one solution of the minimization problem \eqref{eq:JC}.
\end{proposition}
\begin{proof}
Let $\{C_n\}_{n\geq 0}\in\mathcal{C}$ be a minimizing sequence. Then there exists a positive constant $M$ such that
\begin{equation}
    J(C_n)\leq M,\qquad \forall n,
\end{equation}
hence
\begin{equation*}
\textrm{Per}(C_n) \leq M,\qquad \forall n.    
\end{equation*}
By compactness (see Thereom 3.39 in \cite{AFP2000}), there exists a set of finite perimeter $C_0$ such that, possibly up to a subsequence,
\begin{equation*}
 |C_n\triangle C_0|\rightarrow 0,\qquad n\rightarrow \infty,
\end{equation*}
where $C_n\triangle C_0$ is the symmetric difference of the two sets.
Moreover, thanks to the compactness and equiboundedness of the sets $C_n$ and the fact that $C_n\in\mathcal{C}$, there exists a further subsequence which converges in the Hausdorff metric to $C_0\in\mathcal{C}$, thanks to \cite[Theorem 2.4.10]{HenrPierr18}. Moreover, by the lower semicontinuity of the perimeter functional (see Section 5.2.1, Theorem 1, in \cite{EG15})  it follows that
\begin{equation*}
\textrm{Per}(C_0)\leq \liminf_{n\rightarrow\infty}\textrm{Per}(C_n).
\end{equation*}
Using the continuity of the boundary functional, see \eqref{eq:continuity}, we also have
\begin{equation*}
  \int_{\Sigma_N}(u(C_n) - u_{meas})^2\,d\sigma(x)\rightarrow  \int_{\Sigma_N}(u(C_0) - u_{meas})^2\,d\sigma(x),\qquad\textrm{as }\ n\rightarrow\infty.
\end{equation*}
In conclusion, we find that 
\begin{equation*}
J(C_0)\leq \liminf_{n\rightarrow\infty}J(C_n)=\lim_{n\rightarrow\infty}J(C_n)=\inf_{C\in \mathcal{C}}J(C),
\end{equation*}
and the claim follows.
\end{proof}
We also prove stability with respect to the measured data.
\begin{proposition}\label{prop:stability}
Solutions of \eqref{eq:JC} are stable with respect to perturbations of the data $u_{meas}$, i.e., if $u_n\rightarrow u_{meas}$ in $L^2(\Sigma_N)$ as $n\rightarrow \infty$ then the solutions $C_n$ of \eqref{eq:JC} with datum $u_n$ are such that, up to subsequences,
\begin{equation*}
d_H(C_n,\widetilde{C})\rightarrow 0,\,\, \textrm{ as }\ n\rightarrow \infty,
\end{equation*}
where $\widetilde{C}\in \mathcal{C}$  is a solution of \eqref{eq:JC}, with datum $u_{meas}$.
\end{proposition}
\begin{proof}
Using \eqref{eq:JC}, we have that, for any $n$, $C_n$ satisfies
\begin{equation*}
\frac{1}{2} \int_{\Sigma_N}(u(C_n) - u_n)^2\,d\sigma(x)+\alpha \textrm{Per}(C_n)\leq \frac{1}{2} \int_{\Sigma_N}(u(C) - u_n)^2d\sigma(x)+\alpha \textrm{Per}(C),
\end{equation*}
for all $C\in\mathcal{C}$.
Therefore, $\textrm{Per}(C_n)\leq M$ and hence, possibly up to subsequences, 
\begin{equation*}
d_H(C_n,\widetilde{C})\rightarrow 0,\qquad n\rightarrow \infty,
\end{equation*}
for some $\widetilde{C}\in\mathcal{C}$, and
\begin{equation*}
\textrm{Per}(\widetilde{C})\leq \liminf_{n\rightarrow\infty}\textrm{Per}(C_n).
\end{equation*}
Moreover, by the continuity of the solution of \eqref{eq:weak_form_prob_cavity} with respect to $C$, see Theorem \ref{th:continuity}, we get
\begin{equation*}
\begin{aligned}
J(\widetilde{C})&\leq \liminf_{n\rightarrow\infty}\frac{1}{2} \int_{\Sigma_N}(u(C_n) - u_n)^2d\sigma(x)+\alpha \textrm{Per}(C_n)\\
&\leq  \lim_{n\rightarrow\infty}\frac{1}{2} \int_{\Sigma_N}(u(C) - u_n)^2d\sigma(x)+\alpha \textrm{Per}(C)\\
&=\frac{1}{2} \int_{\Sigma_N}(u(C) - u_{meas})^2d\sigma(x)+\alpha \textrm{Per}(C),
\end{aligned}
\end{equation*}
for all $C\in \mathcal{C}$. Summarizing, $\widetilde{C}\in \mathcal{C}$ and it is a minimizer of the functional, hence the assertion follows. 
\end{proof}
Finally, we can prove that the solution of the minimization problem \eqref{eq:JC} converges to the unique solution of the inverse problem when the regularization parameter tends to zero.
\begin{proposition}
Let us assume that there exists a solution $C^{\sharp}\in \mathcal{C}$ of the inverse problem corresponding to datum $u_{meas}$. Moreover, for any $\eta>0$ let  $(\alpha(\eta))_{\eta>0}$ be  such that  $\alpha(\eta)=o(1)$ and $\frac{\eta^2}{\alpha(\eta)}$ is bounded as $\eta\rightarrow 0$.\\
Furthermore, let $C_{\eta}$ be a solution to the minimization problem \eqref{eq:JC} with $\alpha=\alpha(\eta)$ and datum $u_{\eta}\in L^2(\Sigma_N)$ satisfying $\|u_{meas}-u_{\eta}\|_{L^2(\Sigma_N)}\leq \eta$. Then
\begin{equation*}
C_{\eta}\rightarrow C^{\sharp}
\end{equation*}
in the Hausdorff metric, as $\eta\rightarrow 0$.
\end{proposition}

\begin{proof}
From the definition of $C_{\eta}$, it immediately follows that
\begin{equation}\label{eq:eq8}
\begin{aligned}
\frac{1}{2} \int_{\Sigma_N}(u(C_{\eta}) - u_{\eta})^2d\sigma(x)+\alpha\textrm{Per}(C_{\eta})&\leq \frac{1}{2} \int_{\Sigma_N}(u(C^{\sharp}) - u_{\eta})^2d\sigma+\alpha \textrm{Per}(C^{\sharp})\\
&=\frac{1}{2} \int_{\Sigma_N}(u_{meas} - u_{\eta})^2d\sigma+\alpha \textrm{Per}(C^{\sharp})\\
&\leq \eta^2+\alpha \textrm{Per}(C^{\sharp}).
\end{aligned}
\end{equation}
Straightforwardly, we find that
\begin{equation}\label{eq:eq9}
 \textrm{Per}(C_{\eta})\leq \frac{\eta^2}{\alpha}+\textrm{Per}(C^{\sharp})\leq M.
\end{equation}
Hence, up to subsequences, arguing as in Proposition \ref{prop:stability}, we get
\begin{equation*}
d_H(C_{\eta}, C_0)\rightarrow 0,\qquad \textrm{as}\ \eta\rightarrow 0,
\end{equation*}
for some $C_0\in \mathcal{C}$. From \eqref{eq:eq8} and \eqref{eq:eq9}, as $\eta\rightarrow 0$, we find
\begin{equation*}
 \int_{\Sigma_N}(u(C_{\eta}) - u_{\eta})^2d\sigma\rightarrow 0,
\end{equation*}
hence, also
\begin{equation*}
 \int_{\Sigma_N}(u(C_{\eta}) - u_{meas})^2d\sigma(x)\leq \int_{\Sigma_N}(u(C_{\eta}) - u_{\eta})^2d\sigma(x)+\int_{\Sigma_N}(u_{meas}- u_{\eta})^2d\sigma(x)\rightarrow 0.
\end{equation*}
By the continuity result in Theorem \ref{th:continuity} and using the last relation, we find that
\begin{equation*}
 u(C_0)=u_{meas},\qquad\text{on}\ \Sigma_N.
\end{equation*}
Therefore, thanks to the uniqueness result of the inverse problem in Lipschitz domains (cf. \cite{MR04}) we get $C_0=C^{\sharp}$.
\end{proof}
\section{Reconstruction of cavities - filling the void}\label{sec:filling the void}
From the numerical point of view, the minimization of the functional \eqref{eq:JC} is complicated due to its non-differentiability. A typical approach to overcome this issue is to consider a further regularization of the functional, where the perimeter is approximated by a Ginzburg-Landau type functional, see for example \cite{BouCham03}. This approach is well-known in the literature and it has been applied in different contexts, see for example \cite{AS21,ABCHDRR20,BRV2018,BGHFS14,BGHHR16,BouCham03,CRBHRA19,DEV2016,GHHHK15,JZ10,LamYou20}.

First, we note that Problem \eqref{eq:JC} is equivalent to the following formulation
\begin{equation}\label{eq:Jv}
\min_{v\in X_{0,1}}J(v), \hbox{ where } J(v)=\frac{1}{2}\int_{\Sigma_N} |u(v)-u_{meas}|^2\, d\sigma(x)+\alpha TV(v),	
\end{equation}
where 
$X_{0,1}:=\{v\in BV(\Omega)\,:\, v=\chi_{C} \, \hbox{ a.e. in }\Omega, \,C\in {\mathcal C}\}$, $TV(v)$ is defined in \eqref{eq:TV}, and $\chi_C$ is the indicator function of $C$. Note that the space $X_{0,1}$ is endowed with the norm $\|{v}\|_{BV(\Omega)} = \|{v}\|_{L^1(\Omega)} + TV(v)$.
\begin{remark}\label{compactness}
By compactness properties of $BV(\Omega)$ (see, e.g., \cite{AFP2000}, Theorem 3.23), any uniformly bounded sequence in $X_{0,1}$ admits a subsequence  converging in $L^1(\Omega)$ to an element in $X_{0,1}$. In fact, let $v_n$ a sequence uniformly bounded in $X_{0,1}$, there exists, possibly up to a subsequence, $v\in BV(\Omega)$ such that 
\begin{equation*}
    v_n\to v\ \ \textrm{in}\ \ L^1(\Omega) \Rightarrow v_n\to v \ \ \textrm{a.e. in}\ \ \Omega.
\end{equation*}
Since $v_n$ attains values $0$ and $1$ only, it follows that $v\in X_{0,1}$.
\end{remark}
Following the approach proposed in \cite{BouCham03}, we fill the cavity with a fictitious material with elastic properties that are different from the background. Specifically, we take an elasticity tensor $\mathbb{C}_1:=\delta\mathbb{C}_0$, where $\delta>0$ is sufficiently small.
Therefore, the boundary value problem \eqref{prob:cavity} is modified into 
\begin{equation}
	\left\{
	\begin{aligned}
		\textrm{div}(\mathbb{C}_{\delta}(v) \widehat{\nabla} u_{\delta}(v)) &= 0 \qquad \text{in}\ \Omega,	\\
		(\mathbb{C}_{\delta}(v) \widehat\nabla u_{\delta}(v)) \nu &= g \qquad \text{on } \Sigma_N,\\
		u_{\delta}(v)&=0 \qquad \text{on } \Sigma_D,
	\end{aligned}
	\right.
	\label{elasticity}
\end{equation}
where
\begin{equation}\label{eq:elasticity_tensor_inclusion}
\mathbb{C}_{\delta}(v)= \mathbb{C}_0 +  (\mathbb{C}_1 - \mathbb{C}_0)v,\quad \textrm{with}\quad \mathbb{C}_1=\delta\mathbb{C}_0.
\end{equation}
Here $\mathbb{C}_0$ and $\mathbb{C}_1$ are the elasticity tensors in $\Omega \setminus C$ and $C$, respectively.
\begin{remark}
Thanks to Assumption \ref{ass:elasticity_tensor}, the fact that $\delta>0$, and by \eqref{eq:elasticity_tensor_inclusion}, the elasticity tensor $\mathbb{C}_{\delta}(v)$ is strongly convex.
\end{remark}
\begin{remark}
The following analysis can be generalized to the case of a generic fourth-order elasticity tensor $\mathbb{C}_1$ which is strongly convex and uniformly bounded with the further hypothesis that
\begin{equation*}
\mathbb{C}_1\widehat{A}:\widehat{A} \leq \mathbb {C}_0\widehat{A}:\widehat{A} \ \ \text{ or }\ \, \mathbb {C}_0\widehat{A}:\widehat{A} \leq \mathbb{C}_1\widehat{A}:\widehat{A}.
\end{equation*}
\end{remark}
\begin{remark}
When dealing with sequences, we will often use the simplified notation $u_n:= u_{\delta}(v_n), \,  u := u_{\delta}(v), \, \mathbb{C}_n := \mathbb{C}_{\delta}(v_n), \,\mathbb{C} := \mathbb{C}_{\delta}(v)$.
\end{remark}
The elastic problem \eqref{elasticity} has the following weak formulation:
\begin{center}
    \textrm{Find} $u_{\delta}(v)\in H^1_{\Sigma_D}(\Omega)$ \textrm{solution to}
    \begin{equation}\label{eq:weak_formulation_inclusion}
        \int_{\Omega} \mathbb{C}_{\delta}(v)\widehat{\nabla}{u_{\delta}(v)}:\widehat{\nabla}{\varphi}\, dx=\int_{\Sigma_N}g\cdot \varphi\, d\sigma(x), \qquad \forall \varphi\in H^1_{\Sigma_D}(\Omega).
    \end{equation}
\end{center}
Well-posedness of Problem \eqref{elasticity} in $H^1_{\Sigma_D}(\Omega)$ follows in the same way as for Problem \eqref{prob:cavity}, and, in addition 
\begin{equation*}
    \|u_{\delta}(v)\|_{H^1(\Omega)}\leq c \|g\|_{L^2_{\Sigma_N}}.
\end{equation*}
We now approximate Problem \eqref{eq:Jv} with the following one
\begin{equation}\label{eq:Jv_delta}
\min_{v\in X_{0,1}}J_{\delta}(v), \hbox{ where } J_{\delta}(v)=\frac{1}{2}\int_{\Sigma_N} |u_{\delta}(v)-u_{meas}|^2\, d\sigma(x)+\alpha TV(v),	
\end{equation}
where $u_{\delta}(v)\in H^1_{\Sigma_D}(\Omega)$ is the solution of Problem \eqref{elasticity}.\\
We prove the existence of minima of $J_{\delta}(v)$ in $X_{0,1}$, on account of the ideas contained in \cite{BRV2018}. 
The proof is a consequence of the following property.
\begin{proposition}\label{continuity}
Let $\{v_n\} \subset X_{0,1}$ be strongly convergent in $L^1(\Omega)$ to $v \in X_{0,1}$. Then $\{u_{\delta}(v_n)\lfloor_{\Sigma_N}\}$ strongly converges
in $L^2(\Sigma_N)$ to $u_{\delta}(v)\lfloor_{\Sigma_N}$, i.e., the map $F : v \to u_{\delta}(v)\lfloor_{\Sigma_N}$ is continuous from $X_{0,1}$ to $L^2(\Sigma_N)$ in the $L^1$ topology.
\end{proposition}
\begin{proof}
Consider the weak formulation \eqref{eq:weak_formulation_inclusion} associated to $v$ and $v_n$, respectively, 
\begin{equation*}
\int_\Omega\mathbb{C}_{\delta}(v) \widehat\nabla u_{\delta}(v): \widehat\nabla \varphi = \int_{\Sigma_N} g\cdot\varphi, \quad \forall \varphi \in H_{\Sigma_D}^1(\Omega),
\end{equation*}
\begin{equation*}
\int_\Omega\mathbb{C}_{\delta}(v_n) \widehat\nabla u_{\delta}(v_n): \widehat\nabla \varphi = \int_{\Sigma_N} g\cdot\varphi, \quad \forall \varphi \in H_{\Sigma_D}^1(\Omega).
\end{equation*}
Subtracting the two equations and setting $u_n:= u_{\delta}(v_n), \,  u := u_{\delta}(v), \, \mathbb{C}_n := \mathbb{C}_{\delta}(v_n), \,\mathbb{C} := \mathbb{C}_{\delta}(v)$, we get
\begin{equation*}
\int_\Omega\mathbb{C}_n \widehat\nabla (u_n -  u) : \widehat\nabla \varphi + \int_{\Omega}( \mathbb{C}_n - \mathbb{C}) \widehat\nabla u :\widehat\nabla \varphi =0, \quad \forall \varphi \in H_{\Sigma_D}^1(\Omega).
\end{equation*}
Thus, making the choice $\varphi = u_n - u$ and proceeding similarly as in \eqref{eq:H1_estimate} to get $H^1$-estimates, we find
\begin{equation*}
 \|u_n-u\|^2_{H^1(\Omega)}\leq c \|(\mathbb{C}_n - \mathbb{C})\widehat{\nabla} u\|_{L^2(\Omega)}
\|\widehat{\nabla} (u_n - u)\|_{L^2(\Omega)},
\end{equation*}
and then, by $\mathbb{C}_n - \mathbb{C} = (\mathbb{C}_1 - \mathbb{C}_0)(v_n - v)$ and the uniform bound on the elasticity tensor, see Assumption \ref{ass:elasticity_tensor}, we derive
\begin{equation*}
\|u_n-u\|_{H^1(\Omega)} \leq c \|(\widehat{\nabla} u) (v_n - v)\|_{L^2(\Omega)}.
\end{equation*}
Observe now that $v_n - v \to 0$ in $L^1(\Omega)$ as $n\rightarrow +\infty$ so that, possibly up to a subsequence, $v_n -v \to 0$, a.e. in $\Omega$. Moreover, recalling that $v_n$ and $v$ are bounded and $u \in H^1(\Omega)$, we deduce, by dominated convergence theorem, that
\begin{equation*}
 \|u_n -  u\|_{H^1(\Omega)} \to 0,\quad \text{as } n\rightarrow +\infty.
\end{equation*}
Finally, the trace theorem implies
\begin{equation*}
 \|u_n -  u\|_{L^2(\Sigma_N)} \to 0,\quad \text{as } n\rightarrow +\infty.
\end{equation*}
\end{proof}
\begin{proposition}
$J_{\delta}(v)$ admits a minimum $v \in X_{0,1}$.
\end{proposition}
\begin{proof}
Observe that $J_{\delta}(v)$ is bounded from below, by definition. Moreover, $J_{\delta}(v) \neq + \infty$, for $v\in X_{0,1}$.
So, let $\{v_n\} \subset X_{0,1}$ be a minimizing sequence of  $J_{\delta}(v)$, that is
\begin{equation*}
J_{\delta}(v_n) \to \inf_{v \in X_{0,1}}J_{\delta}(v) = M, \quad \text{as } n\rightarrow +\infty.
\end{equation*}
Then
$$0\leq J_{\delta}(v_n) \leq 2M  \text{ \, and  \, } 0\leq \alpha TV(v_n) \leq 2M.$$
Hence, there exists a positive constant $c$, independent on $n$, such that
\begin{equation}
\|v_n\|_{BV(\Omega)} = \|v_n\|_{L^1(\Omega)} + TV(v_n) \leq c.
\end{equation}
This implies that $v_n$ is uniformly bounded in $X_{0,1}$. Therefore, thanks to Remark \ref{compactness}, there exists $v \in X_{0,1}$ such that $v_n\to v$ in $L^1(\Omega)$.
Due to the lower semicontinuity of $TV(v)$ with respect to the $L^1$-convergence, we have
\begin{equation*}
TV(v) \leq \liminf\limits_{n\to+\infty} TV(v_n)
\end{equation*}
and, using Proposition \ref{continuity}, we get
\begin{equation}
\begin{aligned}
J_{\delta}(v) &= \frac12\int_{\Sigma_N}|u_{\delta}(v) - u_{meas}|^2 + \alpha TV(v) \\
&\leq \liminf\limits_{n\to+\infty} \left ( \frac12 \int_{\Sigma_N}|u_{\delta}(v_n) - u_{meas}|^2 + \alpha TV(v_n) \right )\\
&= \lim\limits_{n\to+\infty} J_{\delta}(v_n) = M.
\end{aligned}
\end{equation}
\end{proof}

%%%%%%%%%%%%%%%%%%%%%%%%%%%%%%%%%%%%%%%%%%%%%%%%%%%%%
%%%%%%%%%%%%%%%%%%%%%%%%%%%%%%%%%%%%%%%%%%%%%%%%%%%%%
%%%%%%%%%%%%%%%%%%%%%%%%%%%%%%%%%%%%%%%%%%%%%%%%%%%%%

\subsection{Phase-field relaxation}
Proceeding as in \cite{DEV2016,BRV2018}, we now consider a phase-field relaxation of the optimization problem \eqref{eq:Jv_delta}. More precisely, we define a minimization problem for a differentiable 
cost functional defined on a convex subspace of $H^1(\Omega)$, namely on the set
\begin{equation*}
\mathcal{K} = \{ v \in H^1(\Omega): \ 0 \leq v(x) \leq 1 \ a.e. \ in \ \Omega, \ v(x) = 0 \ a.e. \ in \ \Omega^{d_0/2}\},
\end{equation*}
where $\Omega^{d_0/2}$ has been defined in \eqref{eq:omega_d0},
and, for every $\varepsilon >0$, we replace the total variation term with the following Modica-Mortola functional.
\begin{problem}
Given $u_{meas} \in L^2(\Sigma_N)$, and $\varepsilon,\delta>0$, find
\begin{equation}
\min_{v \in \mathcal{K}} J_{\delta,\varepsilon}(v), \,\,\, J_{\delta,\varepsilon}(v) := \frac12 \|u_{\delta}(v)-u_{meas}\|_{L^2(\Sigma_N)}^2
+ \widetilde{\alpha} \!\int_{\Omega}\Big( \varepsilon|\nabla v|^2 + \frac{1}{\varepsilon}v(1-v)\Big),\label{minrel}
\end{equation}
$u_{\delta}(v)\in H^1_{\Sigma_D}(\Omega)$ being the solution to \eqref{elasticity}, for $v\in \mathcal{K}$, and $\widetilde{\alpha}=\frac{4}{\pi}\alpha$, where $4/\pi=(2\int_0^1 \sqrt{v(1-v)}\, dv)^{-1}$ is a rescaling parameter, see \cite{ALB96}.
\end{problem}
\begin{remark}
We expect $\Gamma$-convergence of the functional $J_{\delta,\varepsilon}$ to $J$, given in \eqref{eq:Jv}. However, this analysis is involved in the elastic context and is still an open issue that needs a specific accurate study.
\end{remark}
The following result holds
\begin{proposition}
For any $\delta,\varepsilon >0$, Problem \eqref{minrel} admits a solution $v=v_{\delta,\varepsilon} \in {\cal K}$.
\end{proposition}
\begin{proof}
Let us fix $\delta,\varepsilon > 0$ and consider a minimizing sequence $\{ v_n \} \subset \mathcal{K}$ for $J_{\delta,\varepsilon}(v)$ 
(we omit the dependence of $v_n$ on $\delta$ and $\varepsilon$). We have
\begin{equation*}
J_{\delta,\varepsilon}(v_n) \to \inf_{v \in \mathcal{K}}J_{\delta,\varepsilon}(v) = M.
\end{equation*}
Hence, by definition of minimizing sequence, $0\leq J_{\delta,\varepsilon}(v_n) \leq 2M$ independently of $n$, which implies that also $\|{\nabla v_n}\|_{L^2(\Omega)}^2$ 
is bounded.
Moreover, recalling that $v_n \in \mathcal{K}$ and $0\leq v_n (x)\leq 1$ a.e. in $\Omega$, we deduce that $\|{v_n}\|_{L^2(\Omega)}\leq M_1$, with $M_1$ independent of $n$ 
and hence $\|{v_n}\|_{H^1(\Omega)}\leq M_2$, with $M_2$ independent of $n$. Due to the weak compactness of $H^1(\Omega)$, there exists $v \in H^1(\Omega)$ 
such that, possibly up to a subsequence, $ v_n  \rightharpoonup v$ in $H^1(\Omega)$.
Hence, $v_n \rightarrow v$ strongly in $L^2(\Omega)$ and $v_n \rightarrow v$ a.e. in $\Omega$.
Since $v_n(1-v_n)\leq 1/4$, by means of the Lebesgue's dominated convergence theorem, we get
\begin{equation*}
	\int_{\Omega} v_{n}(1-v_{n}) \rightarrow \int_{\Omega} v(1-v).
\end{equation*}
	Moreover, by the lower semicontinuity of the $H^1(\Omega)$ norm with respect to the weak convergence, we obtain
\begin{equation*}
	\begin{aligned}
	\|v\|_{H^1(\Omega)}^2 &\leq \liminf\limits_{n\to+\infty}  \|{v_{n}}\|_{H^1(\Omega)}^2, \\ 
	\|{v}\|_{L^2(\Omega)}^2 + \|{\nabla v}\|_{L^2(\Omega)}^2 &\leq  \lim\limits_{n\to+\infty} \|{v_{n}}\|_{L^2(\Omega)}^2 + \liminf\limits_{n\to+\infty} \|{\nabla v_{n}}\|_{L^2(\Omega)}^2\\
&= \|{v}\|_{L^2(\Omega)}^2 + \liminf\limits_{n\to+\infty} \|{\nabla v_{n}}\|_{L^2(\Omega)}^2,\\ 
	\|{\nabla v}\|_{L^2(\Omega)}^2 &\leq \liminf\limits_{n\to+\infty} \|{\nabla v_{n}}\|_{L^2(\Omega)}^2.
	\end{aligned}
\end{equation*}
By the last inequality and the convergence of $v_n$ to $v$ a.e., by the use of Proposition \ref{continuity} and the fact that $v_n$ is a minimizing sequence, we have
\begin{equation*}
\begin{aligned}
J_{\delta,\varepsilon}(v) &= \frac12 \|u_{\delta}(v)-u_{meas}\|_{L^2(\Sigma_N)}^2
+ \widetilde{\alpha} \int_{\Omega}\left( \varepsilon|\nabla v|^2 + \frac{1}{\varepsilon}v(1-v)\right)\\
&\leq \liminf\limits_{n\to+\infty}\left ( \frac12 \|u_{\delta}(v_n)-u_{meas}\|_{L^2(\Sigma_N)}^2
+ \widetilde{\alpha} \int_{\Omega}\left( \varepsilon|\nabla v_n|^2 + \frac{1}{\varepsilon}v_n(1-v_n)\right) \right )\\
&=\lim\limits_{n\to+\infty}	J_{\delta,\varepsilon}(v_n) = M.
\end{aligned}
\end{equation*}
Finally, by pointwise convergence, we know that $0 \leq v \leq 1$ a.e. in $\Omega$ and $v = 1$ a.e. in $\Omega^{d_0/2}$.
Hence, $v$ is a minimum of $J_{\delta,\varepsilon}$ in $\mathcal{K}$.
\end{proof}

%%%%%%%%%%%%%%%%%%%%%%%%%%%%%%%%%%%%%%%%%%%%%%%%%%%%%%%%%%%
%%%%%%%%%%%%%%%%%%%%%%%%%%%%%%%%%%%%%%%%%%%%%%%%%%%%%%%%%%%

\subsection{Necessary optimality conditions}
\label{secop}
In this section we provide an expression for the first order necessary optimality condition associated with the minimization problem \eqref{minrel},
formulated as a variational inequality involving the Fr\'{e}chet derivative of $J_{\delta,\varepsilon}$.
\begin{proposition}\label{Prop:OC}
Define the map $F:\mathcal{K} \rightarrow H^1(\Omega), \, F(v)  = u_{\delta}(v)$, $u_{\delta}(v)$ solution to \eqref{elasticity}.
Then the operators $F$ and $J_{\delta,\varepsilon}$ (for every $\delta, \varepsilon>0$) are Fr\'{e}chet-differentiable on $\mathcal{K} \subset L^{\infty}(\Omega) \cap H^1(\Omega)$.

Moreover, any minimizer $v_{\delta,\varepsilon}$  of $J_{\delta,\varepsilon}$ satisfies the variational inequality
	\begin{equation}
	J'_{\delta,\varepsilon}(v_{\delta,\varepsilon})[\omega - v_{\delta,\varepsilon}] \geq 0, \qquad \forall \omega \in \mathcal{K},
	\label{OC}
	\end{equation}
	where
	\begin{equation}
		J'_{\delta,\varepsilon}(v)[\vartheta] = \int_\Omega\vartheta(\mathbb{C}_0 - \mathbb{C}_1) \widehat\nabla u_{\delta}(v) : \widehat\nabla p_{\delta}(v)
+  2 \widetilde{\alpha} \varepsilon \int_{\Omega}\widehat\nabla v : \widehat\nabla \vartheta
+ \frac{\widetilde{\alpha}}{\varepsilon}\int_{\Omega}(1-2v)\vartheta.
	\label{VI}
	\end{equation}
Here $\vartheta \in \mathcal{K}-v=\{z \ s.t. \ z+v \in \mathcal{K}\}$ and $p_{\delta}\in H_{\Sigma_D}^1(\Omega)$ is the solution to the \textit{adjoint problem}
\begin{equation}
	\int_{\Omega} \mathbb{C}_{\delta}(v) \widehat\nabla p_{\delta} : \widehat\nabla \psi  = \int_{\Sigma_N}{(u_{\delta}(v)-u_{meas})\psi}, \qquad \forall \psi \in H_{\Sigma_D}^1(\Omega).
\label{adjoint}
\end{equation}
\end{proposition}
\begin{proof}
First we prove that $F$ is Fr\'{e}chet differentiable in $L^{\infty}(\Omega)$. More precisely,
$$F'(v)[\vartheta] = u^{\sharp}(v), \text{ for }\vartheta \in L^{\infty}(\Omega)\cap (\mathcal{K}-v),$$
where $u^{\sharp}(v)$ is the solution in $H_{\Sigma_D}^1(\Omega)$ of
\begin{equation}
	\int_{\Omega} \mathbb{C}_{\delta}(v)\widehat\nabla u^{\sharp}(v):\widehat\nabla \varphi  = \int_{\Omega} \vartheta(\mathbb{C}_0 - \mathbb{C}_1)\widehat\nabla u_{\delta}(v) : \widehat\nabla \varphi,
\quad \forall \varphi \in H_{\Sigma_D}^1(\Omega),
	\label{derprob}
\end{equation}
namely,
\begin{equation}
	\|{F(v+\vartheta) - F(v) - u^{\sharp}(v)}\|_{H^1(\Omega)} = o(\|{\vartheta}\|_{L^{\infty}(\Omega)}).
\label{Frechet}
\end{equation}
To this aim, we first show that
$$ \|{u_{\delta}(v+\vartheta) - u_{\delta}(v)}\|_{H^1(\Omega)} \leq c \|{\vartheta}\|_{L^{\infty}(\Omega)}, \, \text{ for \, }\vartheta \in L^{\infty}(\Omega)\cap (\mathcal{K}-v).$$
Indeed, the difference $u_{\delta}(v+\vartheta)-u_{\delta}(v)$ satisfies
\begin{equation}\label{difference}
\begin{aligned}
&\int_{\Omega}\mathbb{C}_{\delta}(v + \vartheta) \widehat\nabla (u_{\delta}(v+\vartheta)-u_{\delta}(v)) : \widehat\nabla \varphi\\
+ &\int_{\Omega}(\mathbb{C}_{\delta}(v + \vartheta) - \mathbb{C}_{\delta}(v)) \widehat\nabla u_{\delta}(v) : \widehat\nabla \varphi  = 0, \quad \forall \varphi \in H_{\Sigma_D}^1(\Omega).
\end{aligned}
\end{equation}
Taking $\varphi = u_{\delta}(v+\vartheta)-u_{\delta}(v)$ and recalling that $\mathbb{C}_{\delta}(v + \vartheta) - \mathbb{C}_{\delta}(v) = (\mathbb{C}_1 - \mathbb{C}_0)\vartheta$,  we obtain
\begin{equation}\label{differencebis}
\begin{aligned}
&\int_{\Omega}\mathbb{C}_{\delta}(v + \vartheta) \widehat\nabla (u_{\delta}(v+\vartheta)-u_{\delta}(v)):\widehat\nabla (u_{\delta}(v+\vartheta)-u_{\delta}(v)) \\
=
- &\int_{\Omega} \vartheta(\mathbb{C}_1 - \mathbb{C}_0)\widehat\nabla u_{\delta}(v) : \widehat\nabla (u_{\delta}(v+\vartheta)-u_{\delta}(v)).
\end{aligned}
\end{equation}
Hence, by using the assumptions on the elasticity tensors, Korn and Poincar\'e inequalities,
and the fact that $v + \vartheta \in \mathcal{K}$, we obtain
\begin{equation}\label{estdiff}
\begin{aligned}
&\|{u_{\delta}(v+\vartheta) - u_{\delta}(v)}\|_{H^1(\Omega)} \leq c\|{\vartheta}\|_{L^\infty(\Omega)}\|{\widehat{\nabla} u_{\delta}(v)}\|_{L^2(\Omega)} \\
&\leq c\|{\vartheta}\|_{L^\infty(\Omega)}\|{u_{\delta}(v)}\|_{H^1(\Omega)}
\leq c\|{\vartheta}\|_{L^\infty(\Omega)}\|g\|_{L^2(\Sigma_N)}\leq c\|{\vartheta}\|_{L^\infty(\Omega)}.
\end{aligned}
\end{equation}
We now estimate $u_{\delta}(v+\vartheta)-u_{\delta}(v) - u^{\sharp}(v)$. Subtracting \eqref{derprob} from \eqref{difference} and setting $\omega= u_{\delta}(v+\vartheta)-u_{\delta}(v)$, then it holds
\begin{equation}
\int_{\Omega}\mathbb{C}_{\delta}(v + \vartheta) \widehat\nabla \omega : \widehat\nabla \varphi
- \int_{\Omega}\mathbb{C}_{\delta}(v) \widehat\nabla u^{\sharp}(v) : \widehat\nabla \varphi  = 0,
\end{equation}
from which
\begin{equation}
\begin{aligned}
\int_{\Omega}\mathbb{C}_{\delta}(v) \widehat\nabla (\omega - u^{\sharp}(v)) : \widehat\nabla \varphi
&= - \int_{\Omega}(\mathbb{C}_{\delta}(v + \vartheta) - \mathbb{C}_{\delta}(v))\widehat\nabla \omega : \widehat\nabla \varphi\\
&= \int_{\Omega}\vartheta(\mathbb{C}_0 - \mathbb{C}_1)\widehat\nabla \omega : \widehat\nabla \varphi.
\end{aligned}
\label{difference3}
\end{equation}
Choosing now $\varphi = \omega - u^{\sharp}(v)$, we get
\begin{equation}
\int_{\Omega}\mathbb{C}_{\delta}(v)\widehat\nabla (\omega - u^{\sharp}(v)):\widehat\nabla (\omega - u^{\sharp}(v))
= \int_{\Omega}(\mathbb{C}_0 - \mathbb{C}_1)\vartheta\widehat\nabla \omega : \widehat\nabla (\omega - u^{\sharp}(v)),
\label{difference4}
\end{equation}
and again by the boundedness of the elasticity tensors and the use of Korn and Poincar\'e inequalities it follows
\begin{equation}
\|\omega - u^{\sharp}(v)\|_{H^1(\Omega)}  = \|u(v+ \vartheta) - u_{\delta}(v) - u^{\sharp}(v)\|_{H^1(\Omega)} \leq c \|\vartheta\|^2_{L^\infty(\Omega)},
\label{difference6}
\end{equation}
so that $F^\prime(v)[\theta] = u^{\sharp}(v)$.

We now prove that $J_{\delta,\varepsilon}$ is Fr\'{e}chet differentiable. By means of the chain rule and the Frech\'{e}t differentiability of $F$,
we compute the expression of $J'_{\delta,\varepsilon}(v)$, i.e.,
\begin{equation}
	J'_{\delta,\varepsilon}(v)[\vartheta] = \int_{\Sigma_N} (F(v)-u_{meas}) F'(v)[\vartheta]
+ \widetilde{\alpha} \int_{\Omega}\Big( 2\varepsilon \nabla v : \nabla \vartheta + \frac{1}{\varepsilon}(1-2v)\vartheta \Big ),
\label{first}
\end{equation}
where, with abuse of notation, $F(v)$ and $F'(v)[\vartheta]$ denote the trace of $F(v)$ and $F'(v)[\vartheta]$ on $\Sigma_N$, respectively.
By the definition of the adjoint problem and of $u^{\sharp}(v)$, we get
\begin{equation}
\begin{aligned}
\int_{\Sigma_N} (F(v)-u_{meas}) F'(v)[\vartheta] =& \int_{\Sigma_N} (F(v)-u_{meas})u^{\sharp}(v)
 =  \\
=& \int_{\Omega}(\mathbb{C}_0 - \mathbb{C}_1)\vartheta \widehat\nabla F(v) : \widehat\nabla p_{\delta}
\end{aligned}
\end{equation}
and hence
\[
 J_{\delta,\varepsilon}'(v)[\vartheta] = \int_{\Omega}(\mathbb{C}_0 - \mathbb{C}_1)\vartheta \widehat\nabla F(v): \widehat\nabla p_{\delta}
 + \widetilde{\alpha} \int_{\Omega}\left( 2\varepsilon \nabla v \cdot \nabla \vartheta + \frac{1}{\varepsilon}(1-2v)\vartheta \right).
\]
Finally, by standard arguments, since $J_{\delta,\varepsilon}$ is a continuous and Frech\'{e}t differentiable functional on a convex subset $\mathcal{K}$ of the Banach space $H^1(\Omega)$,
the optimality conditions for the optimization problem \eqref{minrel} are expressed in terms of the variational inequality \eqref{OC}.
\end{proof}

\section{Discretization and reconstruction algorithm}
\label{sec:discretization} 

\subsection{Convergence analysis}
Here, we assume that $\Omega$ is a polygonal ($d=2$) or polyhedral ($d=3$) domain. 
Again, for simplifying the notation, we denote by $u:=u_{\delta}$ and $p:=p_{\delta}$.

Let $(\mathcal{T}_h)_{0<h\leq h_0}$ be a regular triangulation of $\Omega$ and define 
\begin{equation}\label{calV}
\mathcal{V}_h := \{ v_h \in C(\overline\Omega) \, : \, v_h|_\mathcal{T} \in \mathcal{P}_1(\mathcal{T}), \, \forall \,  \mathcal{T} \in \mathcal{T}_h\}, 
    \end{equation}
where $\mathcal{P}_1(\mathcal{T})$ is the set of polynomials of first degree on $\mathcal{T}$, and
\begin{equation}\label{calKh}
\mathcal{K}_h :=  \mathcal{V}_h \cap\mathcal{K} , 
\quad 
\mathcal{V}_{h,\Sigma_D}:=\mathcal{V}_{h}\cap H^1_{\Sigma_D}(\Omega).
\end{equation}
For every $h>0$, we set $u_h:=(u_{\delta})_h:\mathcal{K}\to \mathcal{V}_{h,\Sigma_D}$ where $u_h$ is solution to 
\begin{equation}\label{eq:u_h}
\int_\Omega\mathbb{C}_{\delta}(v) \widehat\nabla u_h(v): \widehat\nabla \varphi_h = \int_{\Sigma_N} g_h\cdot\varphi_h, \quad \forall \varphi_h \in \mathcal{V}_{h,\Sigma_D}.
\end{equation}
Here $g_h$ is a piecewise linear, continuous approximation of $g$ such that $g_h\to g$ in $L^2(\Sigma_N)$ as $h\to 0$. \\

As in \cite{DEV2016}, one can show that for every 
$v\in\mathcal{K}$ there exists a sequence $v_h\in \mathcal{K}_h$ such that $v_h \to v$ in $H^1(\Omega)$. Most of the following results are an adaptation of those presented in \cite{DEV2016} for a scalar equation to the case of the elasticity system, hence we do not provide the proofs for some of them.

The following lemma is a consequence of the continuity and coercivity of the bilinear form on the left-hand side of \eqref{eq:u_h} and C\'{e}a's Lemma (see, e.g., \cite{BRV2018}).
\begin{lemma}
Let $g\in L^2(\Sigma_N)$. Then, $\forall\, v\in \mathcal{K}$, $u_h(v) \to u(v)$ strongly in $H^1(\Omega)$ as $h\to 0$.
\end{lemma}
Next we state a result concerning the continuity of $u_h$ in the space 
$\mathcal{V}_{h,\Sigma_D}$. 
\begin{proposition}\label{Proposition:5.1}
Let $h_k$, $v_{h_k}$ be two sequences such that 
$\lim\limits_{k\to +\infty} h_k=0$ and $v_{h_k}\in \mathcal{K}_{h_k}$ with $v_{h_k}\to v$ in $L^1(\Omega)$. Then 
$u_{h_k}(v_{h_k})\to u(v)$ in $H^1_{\Sigma_D}(\Omega)$ for $k\to +\infty$.
\end{proposition}
\begin{proof} The proof can be obtained reasoning similarly as in Lemma 3.1 of \cite{DEV2016}.
\end{proof}
Let $J_{\delta,\varepsilon,h}:\mathcal{K}_h\to \mathbb{R}$ be the approximation to $J_{\delta,\varepsilon}$ defined as follows
\begin{equation}\label{Jepsh}
J_{\delta,\varepsilon,h}:=\frac{1}{2}\| u_h(v_h)-u_{meas,h}\|^2_{L^2(\Sigma_N)}+\widetilde{\alpha}\int_{\Omega}\varepsilon
|\nabla v_h|^2+\frac{1}{\varepsilon}v_h(1-v_h),
\end{equation}
where we assume that $u_{meas,h}\to u_{meas}$, as $h\to 0$.
Similarly as in Theorem 3.2 of \cite{DEV2016}, we can show the following result.
\begin{theorem}
There exists $v_h\in \mathcal{K}_h$ such that 
$J_{\delta,\varepsilon,h}(v_h)=\min_{\eta_h\in \mathcal{K}_h}
J_{\delta,\varepsilon,h}(\eta_h)$. Moreover, let $h_k$ be such that $\lim_{k\to +\infty} h_k=0$. Then every sequence $v_{h_k}$ has a subsequence converging strongly in $H^1(\Omega)$ and a.e. in $\Omega$ to a minimum 
of $J_{\delta,\varepsilon}$.
\end{theorem}

In our numerical algorithm we approximately solve \eqref{minrel} and so we  look for an admissible point $v_h\in \mathcal{K}_h$ that satisfies the first order necessary condition 
\begin{equation}
	J'_{\delta,\varepsilon,h}(v_h)[\omega_h - v_h] \geq 0, \qquad \forall \omega_h \in \mathcal{K}_h,
	\label{OCh}
	\end{equation}
rather than trying to locate a global minimum of $J_{\delta,\varepsilon,h}$. To this aim, we consider the discrete adjoint problem: find $p_h:=(p_{\delta})_h\in \mathcal{V}_{h,\Sigma_D}$ such that 
\begin{equation}
	\int_{\Omega} \mathbb{C}_{\delta}(v_h) \widehat\nabla p_h : \widehat\nabla \psi_h  = \int_{\Sigma_N}{(u_h(v_h)-u_{meas,h})\psi_h}, \qquad \forall \psi_h \in \mathcal{V}_{h,\Sigma_D}.
\label{discradjoint}
\end{equation}
Then using $v_h\in \mathcal{K}_h$, we can prove the discrete version of Proposition \ref{Prop:OC}, where the discrete variational inequality reads as:
\begin{eqnarray}
 &&\int_\Omega(\mathbb{C}_0 - \mathbb{C}_1)(\omega_h-v_h) \widehat\nabla u_h(v_h) : \widehat\nabla p_h
+  2 \widetilde{\alpha} \varepsilon \int_{\Omega}\nabla v_h \cdot \nabla (\omega_h-v_h)\nonumber\\
&&\quad+ \frac{\widetilde{\alpha}}{\varepsilon}\int_{\Omega}(1-2v_h)(\omega_h-v_h)
\geq 0, \quad \forall \omega_h\in \mathcal{K}_h\label{OCh:2}.
\end{eqnarray}
Then, we can prove the following theorem:
\begin{theorem}
Let $h_k$ be such that $\lim_{k\to+\infty} h_k=0$ and 
$v_{h_k}$ be a sequence satisfying \eqref{OCh}. Then there exists a subsequence of $v_{h_k}$ that converges strongly in $H^1(\Omega)$ and a.e. in $\Omega$ to a solution $v$ of \eqref{OC}.
\end{theorem}
\begin{proof}
We set $v_k:=v_{h_k}$, $u_k:=u_{h_k}(v_{h_k})$ and $p_k:=p_{h_k}(v_{h_k})$. Testing \eqref{discradjoint} with $\psi_h=p_k$ we get 
$$	\int_{\Omega} \mathbb{C}_{\delta}(v_k) \widehat\nabla p_k :\widehat{\nabla}p_k = \int_{\Sigma_N}{(u_k-u_{meas,k})\cdot p_k}, $$
which yields, arguing as in \eqref{eq:H1_estimate} to get $H^1$-estimates, 
$$c \|p_k\|^2_{H^1(\Omega)} \leq \|u_k-u_{meas,k}\|_{L^2(\Sigma_N)}\|p_k\|_{L^2(\Sigma_N)}.$$
As the problem for $u_k$ is well-posed with  $u_k\in H^1_{\Sigma_D}(\Omega)$ and $u_{meas,k}\to u_{meas}$ 
(implying that $\|u_{meas,k}\|_{L^2(\Sigma_N)}$ is uniformly bounded with respect to $k$), we get 
$$\| p_k\|_{H^1(\Omega)}\leq c.$$ A similar result holds for $\| u_k\|_{H^1(\Omega)}$. Therefore 
\begin{equation}\label{unif_bound}
    \| p_k\|_{H^1(\Omega)} + \| u_k\|_{H^1(\Omega)} \leq c,\quad \text{uniformly~ in~} k.
\end{equation}
From \eqref{OCh:2}, employing $(1-2v_k)(w_k-v_k)\leq w_k + 2v_k^2$ and testing with $w_k=0\in \mathcal{K}_h$, we get 
\begin{equation}
    2\widetilde{\alpha} \varepsilon \int_\Omega \vert \nabla v_k\vert^2 \leq c \|\widehat\nabla u_k \|_{L^2(\Omega)} \|\widehat\nabla p_k  \|_{L^2(\Omega)} + \frac{2\widetilde{\alpha}}{\varepsilon}\vert \Omega\vert \leq c_{\varepsilon},
\end{equation}
where we used \eqref{unif_bound}.
Therefore, $v_k$ is bounded in $H^1(\Omega)$, hence there exists a subsequence (still denoted by $v_k$) and $v\in\mathcal{K}$ such that 
\begin{eqnarray}
&&v_k\rightharpoonup v\quad\text{in~}H^1(\Omega),\qquad v_k\to v\quad\text{in~}L^2(\Omega) ~~(\text{and~in~}L^1(\Omega)),\nonumber\\
&&v_k\to v \quad\text{a.e.~in~}\Omega.\nonumber
\end{eqnarray}
Thanks to Proposition \ref{Proposition:5.1} we have 
\begin{equation}\label{eq:uktou}
    u_k\to u \quad\text{in~}H^1_{\Sigma_D}(\Omega).
\end{equation}
Now, let $p\in H^1_{\Sigma_D}(\Omega)$ be the solution of the continuous adjoint problem and let $\hat{p}_k\in \mathcal{V}_{h_k,\Sigma_D}$ be such that $\hat{p}_k\to p$ in $H^{1}_{\Sigma_D}(\Omega)$. Taking the difference of the problems solved by $p$ and $p_k$, after some standard manipulation we get 
\begin{eqnarray}
&&\int_{\Omega}\mathbb{C}_{\delta}(v_k)\widehat\nabla(p_k-
\hat{p}_k):\widehat\nabla\psi\nonumber \\
&&= 
\int_{\Omega}\mathbb{C}_{\delta}(v_k)\widehat\nabla(p-
\hat{p}_k):\widehat\nabla\psi+
\int_{\Omega}(\mathbb{C}_{\delta}(v)-\mathbb{C}_{\delta}(v_k))\widehat\nabla p:\widehat\nabla\psi\nonumber\\ 
&&\ + 
\int_{\Sigma_N}(u_k-u)\cdot\psi + \int_{{\Sigma}_N}(u_{meas}-u_{meas,k})\cdot\psi,\nonumber
\end{eqnarray}
for all $\psi\in \mathcal{V}_{h_k,\Sigma_D}$.
Taking $\psi=p_k-\hat{p}_k$,  
we get
\begin{eqnarray}
&&\|p_k-\hat{p}_k \|_{H^1(\Omega)}\leq 
c \Big(\|\widehat\nabla(p-\hat{p}_k)\|_{L^2(\Omega)} 
+ \int_{\Omega} \vert v-v_k\vert^2 \vert \widehat\nabla p\vert^2\nonumber\\
&&+\|u_k-u\|_{L^2(\Sigma_N)} + 
\| u_{meas}-u_{meas,k}\|_{L^2(\Sigma_N)}\Big).
\end{eqnarray}
By hypothesis, we have 
$ \| u_{meas}-u_{meas,k}\|_{L^2(\Sigma_N)}\to 0$ and $\| p-\hat{p}_k\|_{H^1(\Omega)}\to 0$ for $k\to+\infty$. Hence, invoking Proposition \ref{Proposition:5.1} and observing that 
$ \int_{\Omega} \vert v-v_k\vert^2 \vert \widehat\nabla p\vert^2 \to 0$ for $k\to+\infty$, we deduce $p_k\to p$ in $H^1(\Omega)$.

Next, we have to show that $v$ satisfies the variational inequality \eqref{OC}.
Given $\omega\in\mathcal{K}$, there exists a sequence 
$\hat{\omega}_k\in\mathcal{K}_{h_k}$ such that 
$\hat{\omega}_k\to \omega$ in $H^1(\Omega)$ and a.e. in $\Omega$. Then, from the discrete variational inequality \eqref{OCh:2} we have for $v_k$ that
\begin{eqnarray}\label{eq:eq12}
 &&\int_\Omega(\mathbb{C}_0 - \mathbb{C}_1)(\hat{\omega}_k-v_k) \widehat\nabla u_k : \widehat\nabla p_k
+  2 \widetilde{\alpha} \varepsilon \int_{\Omega}\nabla v_k \cdot \nabla (\hat{\omega}_k-v_k)\nonumber\\
&&\quad+ \frac{\widetilde{\alpha}}{\varepsilon}\int_{\Omega}(1-2v_k)(\hat{\omega}_k-v_k)\geq 0.
\end{eqnarray}
Now, observe that
\begin{eqnarray}\label{eq:eq10}
&&\int_\Omega(\mathbb{C}_0 - \mathbb{C}_1)(\hat{\omega}_k-v_k) \widehat\nabla u_k : \widehat\nabla p_k-
\int_\Omega(\mathbb{C}_0 - \mathbb{C}_1)(\omega-v) \widehat\nabla u : \widehat\nabla p\nonumber\\
&&=
\int_\Omega(\mathbb{C}_0 - \mathbb{C}_1)(\hat{\omega}_k-v_k) 
[\widehat\nabla (u_k-u): \widehat\nabla p_k + \widehat\nabla u: \widehat\nabla (p_k-p)] \nonumber\\
&&+
\int_\Omega(\mathbb{C}_0 - \mathbb{C}_1)[(\hat{\omega}_k-\omega)-(v_k-v)] \widehat\nabla u:\widehat\nabla p.
\end{eqnarray}
The first integral on the right hand side converges to zero by \eqref{eq:uktou}  and $p_k\to p$ in $H^1(\Omega)$. To show that also the second integral converges to zero, we invoke the dominated convergence theorem. 
Hence, from \eqref{eq:eq10}, we obtain
\begin{equation}\label{eq:eq11}
\int_\Omega(\mathbb{C}_0 - \mathbb{C}_1)(\hat{\omega}_k-v_k) \widehat\nabla u_k : \widehat\nabla p_k-
\int_\Omega(\mathbb{C}_0 - \mathbb{C}_1)(\omega-v) \widehat\nabla u : \widehat\nabla p \to 0,
\end{equation}
as $k\to+\infty$. Then, utilizing \eqref{eq:eq11} into \eqref{eq:eq12}, together with the fact that $v_k \rightharpoonup v$ in $H^1(\Omega)$, and the  lower semicontinuity of the norm, we find 
$$ \|\nabla v\|^2_{L^2(\Omega)}\leq 
\liminf_{k\to+\infty}  \| \nabla v_k\|^2_{L^2(\Omega)}.$$
Morover, noticing that $\int_\Omega v_k \hat{\omega}_k\to \int_\Omega v \omega$ for $k\to+\infty$, we get 
\begin{eqnarray}
&&\int_\Omega(\mathbb{C}_0 - \mathbb{C}_1)(\omega-v) \widehat\nabla u : \widehat\nabla p + 2 \widetilde{\alpha} \varepsilon \int_{\Omega}\nabla v \cdot \nabla (\omega-v) \nonumber \\
&&\quad + \frac{\widetilde{\alpha}}{\varepsilon}\int_{\Omega}(1-2v)(\omega-v)\nonumber\\
&&\geq \liminf_{k\to+\infty}\Big\{
\int_\Omega(\mathbb{C}_0 - \mathbb{C}_1)(\hat{\omega}_k-v_k) \widehat\nabla u_k : \widehat\nabla p_k
 \nonumber\\
&&\quad +  2 \widetilde{\alpha} \varepsilon \int_{\Omega}\nabla v_k \cdot \nabla (\hat{\omega}_k-v_k) + \frac{\widetilde{\alpha}}{\varepsilon}\int_{\Omega}(1-2v_k)(\hat{\omega}_k-v_k)
\Big\}\geq 0.
\end{eqnarray}
Finally, it remains to show that $v_k\to v$ strongly in $H^1(\Omega)$. We choose a sequence 
$\hat{v}_k\in\mathcal{K}_{h_k}$ such that $\hat{v}_k\to v$ in $H^1(\Omega)$ and using the discrete variational inequality \eqref{OCh:2} with $\omega_{h_k}=\hat{v}_k$, we easily get $\nabla v_k\to \nabla v$ in $L^2(\Omega)$, implying the result.
\end{proof}

%%%%%%%%%%%%%%%%%%%%%%%%%%%%%%%%%%%%%%%%%%%%%%
%%%%%%%%%%%%%%%%%%%%%%%%%%%%%%%%%%%%%%%%%%%%%%
%%%%%%%%%%%%%%%%%%%%%%%%%%%%%%%%%%%%%%%%%%%%%%

\subsection{Reconstruction Algorithm}
{In order to solve the discrete optimization problem we follow the method used in \cite{BRV2018} and \cite{DEV2016}. The method is based on solving the following parabolic obstacle problem.
For $\delta,\varepsilon>0$ fixed, let $v$ be the solution to 
\begin{eqnarray}
    && \int_\Omega \partial_t v(\omega-v) + J'_{\delta,\varepsilon}(v)[\omega-v]\geq 0, \quad \forall \omega\in\mathcal{K}, t\in(0+\infty),\nonumber\\
    && v(\cdot,0)=v_0\in\mathcal{K}.\nonumber
\end{eqnarray}
An easy computation shows that the value of the objective functional decreases in time. Hence, we expect that if the limit as $t\to+\infty$ of its solution $v(\cdot,t)$ exists and it is equal to the asymptotic state $v_\infty$, then this should satisfy the continuous optimality conditions \eqref{OC}. }\\
%If the limit as $t\to+\infty$ of its solution $v(\cdot,t)$ is equal to an asymptotic state $v_\infty$, then it should satisfy the continuous optimality conditions \eqref{OC}.
We now discretize the above problem by using a semi-implicit time discretization scheme. We denote by $\{ v_h^n\}_{n\in\mathbb{N}}\subset \mathcal{K}_h$ the sequence of approximations $v_h^n\simeq v(\cdot,t^n)$ obtained as follows:
\begin{eqnarray}
&&v_h^0=v_0\in\mathcal{K}_h\nonumber\\
&&v_h^{n+1}\in\mathcal{K}_h:~
\frac{1}{\tau_n}\int_{\Omega}
(v_h^{n+1}-v_h^n)(\omega_h-v_h^{n+1})\nonumber\\
&&~+
\int_\Omega(\mathbb{C}_0 - \mathbb{C}_1)(\omega_h-v^{n+1}_h) \widehat\nabla u^n_h: \widehat\nabla p^n_h
+  2 \widetilde{\alpha} \varepsilon \int_{\Omega}\nabla v^{n+1}_h \cdot \nabla (\omega_h-v^{n+1}_h)\nonumber\\
&&~+ \frac{\widetilde{\alpha}}{\varepsilon}\int_{\Omega}(1-2v^n_h)(\omega_h-v^{n+1}_h)\geq 0, \quad \forall \omega_h\in \mathcal{K}_h,\ n\geq 0,\label{parab_ineq}
\end{eqnarray}
where $\tau_n$ is the time step, and $u_h^n$, $p_h^n\in \mathcal{V}_{h,\Sigma_D}$ are the discrete solutions of the forward problem \eqref{eq:u_h} and adjoint problem \eqref{discradjoint}, respectively, for $v_h=v_h^n$.
We now prove a monotonicity property of the method.
\begin{lemma}\label{lemma:monotonicity}
For each $n \in \mathbb{N}$, there exists a constant $c^n>0$ such that, if $\tau_n\leq (1+c^n)^{-1}$, then 
\begin{equation}
    \|v_h^{n+1}-v_h^n\|^2_{L^2(\Omega)}+J_{\delta,\varepsilon,h}(v_h^{n+1})\leq J_{\delta,\varepsilon,h}(v_h^n), 
\end{equation}
where 
$c^n=c^n(\Omega,\delta,\xi_0,h,\| \mathbb{C}_0 - \mathbb{C}_1\|_{L^\infty(\Omega)}, \| p_h^n\|_{W^{1,\infty}(\Omega)}, \| u_h^n\|_{W^{1,\infty}(\Omega)}).$
\end{lemma}
\begin{proof}
Choosing $\omega_h=v_h^n$ in \eqref{parab_ineq}, after some simple manipulations we obtain
\begin{eqnarray}
&&\frac{1}{\tau_n} \|v_h^{n+1}-v_h^n\|^2_{L^2(\Omega)}
+\widetilde{\alpha}\varepsilon \|\nabla(v_h^{n+1}-v_h^n)\|^2_{L^2(\Omega)}\nonumber\\
&&+ \frac{\widetilde{\alpha}}{\varepsilon}\|v_h^{n+1}-v_h^n\|^2_{L^2(\Omega)}
+\widetilde{\alpha} \int_{\Omega}\left(\varepsilon
\vert \nabla v_h^{n+1}\vert^2 - \frac{1}{\varepsilon}v_h^{n+1}(1-v_h^{n+1})\right)\nonumber\\
&&-
\widetilde{\alpha} \int_{\Omega}\left(\varepsilon
\vert \nabla v_h^{n}\vert^2 - \frac{1}{\varepsilon}v_h^{n}(1+v_h^{n})\right)\nonumber\\
&&\leq 
\int_\Omega(\mathbb{C}_0 - \mathbb{C}_1)(v_h^n-v^{n+1}_h) \widehat\nabla u_h^n : \widehat\nabla p^n_h.\nonumber
\end{eqnarray}
Adding and subtracting  $\frac{1}{2}\|{u_h^{n+1} - u_{meas,h}}\|_{{L^2(\Sigma_N)}}^2 $ and 
$\frac{1}{2}\|{u_h^n- u_{meas,h}}\|_{{L^2(\Sigma_N)}}^2 $, we get
\begin{eqnarray}
&&\frac{1}{\tau_n} \|v_h^{n+1}-v_h^n\|^2_{L^2(\Omega)}
+\widetilde{\alpha}\varepsilon \|\nabla(v_h^{n+1}-v_h^n)\|^2_{L^2(\Omega)}+ \frac{\widetilde{\alpha}}{\varepsilon}\|v_h^{n+1}-v_h^n\|^2_{L^2(\Omega)}\nonumber\\
&&
+J_{\delta,\varepsilon,h}(v_h^{n+1})-
\frac{1}{2}\|{u_h^{n+1} - u_{meas,h}}\|_{{L^2(\Sigma_N)}}^2
-J_{\delta,\varepsilon,h}(v_h^{n})\nonumber\\
&&+
\frac{1}{2}\|{u_h^n - u_{meas,h}}\|_{{L^2(\Sigma_N)}}^2 \leq 
\int_\Omega(\mathbb{C}_0 - \mathbb{C}_1)(v_h^n-v^{n+1}_h) \widehat\nabla u_h^n : \widehat\nabla p^n_h,\nonumber
\end{eqnarray}
which implies 
\begin{eqnarray}
&&\frac{1}{\tau_n} \|v_h^{n+1}-v_h^n\|^2_{L^2(\Omega)}
+\widetilde{\alpha}\varepsilon \|\nabla(v_h^{n+1}-v_h^n)\|^2_{L^2(\Omega)}+ \frac{\widetilde{\alpha}}{\varepsilon}\|v_h^{n+1}-v_h^n\|^2_{L^2(\Omega)}\nonumber\\
&&
\quad +J_{\delta,\varepsilon,h}(v_h^{n+1})
-J_{\delta,\varepsilon,h}(v_h^{n})\nonumber\\
&&\leq \int_\Omega(v_h^n-v^{n+1}_h)(\mathbb{C}_0 - \mathbb{C}_1) \widehat\nabla u_h^n : \widehat\nabla p^n_h + \frac{1}{2}\|u_h^{n+1} - u_h^n\|_{L^2(\Sigma_N)}^2\nonumber\\
&& \quad +\int_{\Sigma_N} (u_h^{n+1} - u_h^n)\cdot(u_h^n-u_{meas,h})\nonumber\\
&&= \int_\Omega(v_h^n-v^{n+1}_h)(\mathbb{C}_0 - \mathbb{C}_1) \widehat\nabla u_h^n : \widehat\nabla p^n_h + \frac{1}{2}\|u_h^{n+1} - u_h^n\|_{L^2(\Sigma_N)}^2\nonumber\\
&& \quad +\int_\Omega\mathbb{C}_{\delta}(v_h^n)\widehat\nabla p^n_h:
\widehat\nabla (u_h^{n+1}-u_h^n)\nonumber\\
&&=\int_\Omega(\mathbb{C}_{\delta}(v_h^{n+1})-\mathbb{C}_{\delta}(v_h^{n}))\widehat\nabla u^n_h : \widehat\nabla p^n_h + \frac{1}{2}\|u_h^{n+1} - u_h^n\|_{L^2(\Sigma_N)}^2\nonumber\\
&& \quad +\int_\Omega\mathbb{C}_{\delta}(v_h^n)\widehat\nabla p^n_h:
\widehat\nabla (u_h^{n+1}-u_h^n),\label{aux:1}
\end{eqnarray}
where in the last step we employed 
$$ 
\mathbb{C}_{\delta}(v_h^{n+1})-\mathbb{C}_{\delta}(v_h^{n})=(\mathbb{C}_0 - \mathbb{C}_1)(v_h^n-v^{n+1}_h).
$$
It is easy to verify that it holds 
\begin{eqnarray}
&&\int_\Omega(\mathbb{C}_{\delta}(v_h^{n+1})-\mathbb{C}_{\delta}(v_h^{n})) \widehat\nabla u_h^n : \widehat\nabla p^n_h + \frac{1}{2}\|u_h^{n+1} - u_h^n\|_{L^2(\Sigma_N)}^2\nonumber\\
&& \quad +\int_\Omega\mathbb{C}_{\delta}(v_h^n)\widehat\nabla p^n_h:
\widehat\nabla (u_h^{n+1}-u_h^n)\nonumber\\
&&=\int_\Omega(\mathbb{C}_{\delta}(v_h^{n})-\mathbb{C}_{\delta}(v_h^{n+1})) \widehat\nabla(  u_h^{n+1}- u_h^n) : \widehat\nabla p^n_h + \frac{1}{2}\|u_h^{n+1} - u_h^n\|_{L^2(\Sigma_N)}^2\nonumber\\
&&+ \int_\Omega \mathbb{C}_{\delta}(v_h^{n+1})
\widehat\nabla u_h^{n+1}: \widehat\nabla p^n_h -
\int_\Omega \mathbb{C}_{\delta}(v_h^{n})
\widehat\nabla  u_h^n: \widehat\nabla p^n_h\nonumber\\
&&=\int_\Omega(\mathbb{C}_{\delta}(v_h^{n})-\mathbb{C}_{\delta}(v_h^{n+1})) \widehat\nabla(  u_h^{n+1}- u_h^n) : \widehat\nabla p^n_h + \frac{1}{2}\|u_h^{n+1} - u_h^n\|_{L^2(\Sigma_N)}^2\nonumber\\
&&=: I_1,\nonumber
\end{eqnarray}
where the last step follows from the definition of the discrete adjoint problem.

Then, using the Cauchy-Schwarz inequality, the trace theorem and the fact that in finite dimensional spaces all norms are equivalent, we have
\begin{eqnarray}
\vert I_1\vert &\leq& c_0^n \|\mathbb{C}_1 - \mathbb{C}_0\|_{L^\infty(\Omega)} \| \widehat\nabla p_h^n\|_{L^\infty(\Omega)} 
\| v_h^n - v_h^{n+1}\|_{L^2(\Omega)}
\|\widehat\nabla(u_h^{n+1}-u_h^n) \|_{L^2(\Omega)} \nonumber\\
&&+ \frac{1}{2}\|u_h^{n+1} - u_h^n\|_{L^2(\Sigma_N)}^2\nonumber\\
&\leq& c_1^{n} \| v_h^n - v_h^{n+1}\|_{L^2(\Omega)} \|u_h^{n+1}-u_h^n\|_{H^1(\Omega)}+ \frac{c_2^n}{2}\|u_h^{n+1} - u_h^n\|_{H^1(\Omega)}^2\label{eq:intermediate}
\end{eqnarray}
where $c_0^n=c_0^n(\Omega, h)$, $c_1^n=c^n_1(\|\mathbb{C}_1 - \mathbb{C}_0\|_{L^\infty(\Omega)}, \| \widehat\nabla p_h^n\|_{L^\infty(\Omega)}, \Omega,h)$ and $c_2^n$ is the constant in the trace theorem.

In the sequel we bound $\|u_h^{n+1} - u_h^n\|_{H^1(\Omega)}$ by means of the term  
$\| v_h^n - v_h^{n+1}\|_{L^2(\Omega)}$.
To this aim,  we subtract  the equations for $u_h^{n+1}$ and $u_h^n$ (cf. \eqref{eq:u_h}) and employ $\varphi=u_h^{n+1}- u_h^n$ as a test function. A standard manipulation yields 
\begin{equation}\label{aux_bound}
    \|u_h^{n+1} - u_h^n\|_{H^1(\Omega)} \leq c_3^n \| v_h^n - v_h^{n+1}\|_{L^2(\Omega)},
\end{equation}
with $c_3^n=c^n_3(\Omega,\delta,\xi_0,h,\|\mathbb{C}_1 - \mathbb{C}_0\|_{L^\infty(\Omega)}, \| \widehat\nabla p_h^n\|_{L^\infty(\Omega)})$.
Employing \eqref{aux_bound} into \eqref{eq:intermediate}, we obtain 
\begin{equation}\label{bound:I1}
    \vert I_1 \vert \leq c_4^n \| v_h^{n+1} - v_h^n\|^2_{L^2(\Omega)},
\end{equation}
where 
$c_4^n=c^n_4(\Omega,\delta,\xi_0,h,\|(\mathbb{C}_1 - \mathbb{C}_0\|_{L^\infty(\Omega)}, \| \widehat\nabla p_h^n\|_{L^\infty(\Omega)},c_2^n)$.

Using \eqref{bound:I1} into \eqref{aux:1}, we deduce
\begin{eqnarray}
&&\frac{1}{\tau_n} \|v_h^{n+1}-v_h^n\|^2_{L^2(\Omega)}
+\widetilde{\alpha}\varepsilon \|\nabla(v_h^{n+1}-v_h^n)\|^2_{L^2(\Omega)}+ \frac{\widetilde{\alpha}}{\varepsilon}\|v_h^{n+1}-v_h^n\|^2_{L^2(\Omega)}\nonumber\\
&&\quad + J_{\delta,\varepsilon,h}(v_h^{n+1}) \leq J_{\delta,\varepsilon,h}(v_h^{n}) +
 c_4^n \| v_h^{n+1} - v_h^n\|^2_{L^2(\Omega)}.
\end{eqnarray}
Now, since 
\begin{eqnarray}
&&\frac{1}{\tau_n} \|v_h^{n+1}-v_h^n\|^2_{L^2(\Omega)}
+\widetilde{\alpha}\varepsilon \|\nabla(v_h^{n+1}-v_h^n)\|^2_{L^2(\Omega)}+ \frac{\widetilde{\alpha}}{\varepsilon}\|v_h^{n+1}-v_h^n\|^2_{L^2(\Omega)}\nonumber\\
&&\quad \geq 
 \frac{1}{\tau_n} \| v_h^{n+1} - v_h^n\|^2_{L^2(\Omega)},
\end{eqnarray}
we get 
\begin{equation}
    \left(\frac{1}{\tau_n}-c_4^n\right)
    \| v_h^{n+1} -v_h^n\|^2_{L^2(\Omega)} +
    J_{\delta,\varepsilon,h}(v_h^{n+1})\leq J_{\delta,\varepsilon,h}(v_h^{n}).
\end{equation}
Finally, choosing $\tau_n \leq \frac{1}{1+c_4^n}$, the assertion of the lemma follows, just setting $c^n:=c^n_4$.
\end{proof}
We are now ready to state a convergence result for our numerical scheme.
\begin{theorem}
Let $v_h^0\in\mathcal{K}_h$ be an initial guess. Then there exists a collection of timesteps $\tau_n$ such that $0< \gamma \leq \tau_n \leq (1+c^n)^{-1} $, $\forall n>0$, where $c^n$ is the constant appearing in Lemma \ref{lemma:monotonicity}, and $\gamma$ depends on the data and possibly on $h$. The corresponding sequence $v_h^n$ generated by \eqref{parab_ineq} has a convergence subsequence (still denoted by $v_h^n$) in $W^{1,\infty}$ such that 
$$v_h^n\to v_h, \qquad n\to+\infty, $$
where $v_h\in\mathcal{K}_h$ satisfies the discrete optimality condition 
$$J'_{\delta,\varepsilon,h}(v_h)[\omega_h-v_h]\geq 0,\quad \forall \omega_h\in \mathcal{K}_h .$$
\end{theorem}
\begin{proof}
Consider a collection of timesteps bounded by $(1+c^n)^{-1}$, for all $n>0$. Employing Lemma \ref{lemma:monotonicity}, we have 
\begin{eqnarray}
&&\sum_{n=0}^{+\infty} \|v_h^n-v_h^{n+1}\|^2_{L^2(\Omega)}
\leq J_{\delta,\varepsilon,h}(v_h^0),\label{aux:conv:1}\\
&&\sup_{n\in\mathbb{N}} J_{\delta,\varepsilon,h}(v_h^n)\leq J_{\delta,\varepsilon,h}(v_h^0).\label{aux:conv:2}
\end{eqnarray}
Hence, the sequence $v_h^n$ is bounded in $H^1_0(\Omega)$ and it holds 

\begin{equation}\label{aux:conv:3}
    \lim_{n\to +\infty} \|v_h^n-v_h^{n+1}\|^2_{L^2(\Omega)} =0.
\end{equation} 
From the weak formulation of the forward and adjoint problems, the previous relations give that $u_h^n$ and $p_h^n$ are bounded in $H^1(\Omega)$, hence in $W^{1,\infty}(\Omega)$ as we are in finite dimensional spaces. Therefore, thanks to the definition of the constant 
$c^n$, reported in the last part of the proof of Lemma \ref{lemma:monotonicity}, this gives that there exists a constant $M>0$ such that $c^n<M$, and equivalently there exists a positive constant $\gamma>0$, independent of $n$, such that $\gamma\leq (1+c^n)^{-1}$. Hence, there exists a subsequence of 
$(v_h^n,u_h^n,p_h^n)$ (still denoted by the same symbol) such that 
$$ (v_h^n,u_h^n,p_h^n) \to (v_h,u_h,p_h)\quad\text{in~}W^{1,\infty}(\Omega),$$
and in particular 
$$u_h^n\to u_h~\text{a.e.~in~}\Omega, \qquad p_h^n\to p_h~\text{a.e.~in~}\Omega. $$
Hence, $u_h$ is the solution of the discrete forward problem and $p_h$ is the solution of the discrete adjoint problem. 
Finally, from \eqref{parab_ineq} and $\tau_n \geq \gamma$ we get 
\begin{eqnarray}
&&\int_\Omega(\mathbb{C}_0 - \mathbb{C}_1)(\omega_h-v^{n+1}_h) \widehat\nabla u^n_h : \widehat\nabla p^n_h
+  2 \widetilde{\alpha} \varepsilon \int_{\Omega}\widehat\nabla v^{n+1}_h \cdot \widehat\nabla (\omega_h-v^{n+1}_h)\nonumber\\
&&~+ \frac{\widetilde{\alpha}}{\varepsilon}\int_{\Omega}(1-2v^n_h)(\omega_h-v^{n+1}_h) \geq -\frac{C}{\gamma} \|v_h^{n+1}-v_h^n\|_{L^2(\Omega)}
\| \omega_h -v_h^{n+1}\|_{L^2(\Omega)}.\nonumber
\end{eqnarray}
From \eqref{aux:conv:3} and recalling that $v_h^n\to v_h$, we deduce that $v_h$ satisfies the discrete optimality condition \eqref{OCh}.
\end{proof}

\section{Numerical Examples}\label{sec:numerical examples}
In this section we show the numerical results which are obtained from an application of the Primal Dual Active Set Method (PDASM) to the variational inequality \eqref{parab_ineq}. This method has been presented in \cite{HintItoKun03} and later applied for the detection of conductivity inclusions in \cite{DEV2016} and \cite{BRV2018} for a linear and a semilinear elliptic equation, respectively. Primal dual active set methods represent a very good choice in engineering
applications  due to their effectiveness and robustness (cf., e.g., \cite{HePen19}).
Here, we show that choosing the parameter $\delta$ sufficiently small we are able to reconstruct elastic cavities of different shapes. Given a tolerance $\textrm{tol}>0$, the reconstruction algorithm is based on the following steps. 
\begin{algorithm}[H]
\caption{Discrete Parabolic Obstacle Problem}
\label{al:algorithm}
\begin{algorithmic}
\State{Set $n = 0$ and $v_h^0 = v_0$, the initial guess for the inclusion \; }
\While{$\|{v^n_h-v^{n-1}_h}\|>\textrm{tol}$}
\State{find solution of the forward problem \eqref{eq:u_h} with $v = v_h^n$\;}
\State{find solution of the adjoint problem \eqref{discradjoint} with $v = v_h^n$\;}
\State{find $v_h^{n+1}$ solving \eqref{parab_ineq} via PDASM algorithm \;}
\State{update $n = n+1$;}
\EndWhile
\end{algorithmic}
\end{algorithm}
In the implementation of Algorithm \ref{al:algorithm}, the numerical experiments are performed for $d=2$ in the domain $\Omega=(-1,1)^2$, using a triangular tessellation $\mathcal{T}_h$ of $\Omega$. As boundary measurements, we use synthetic data. They are generated by solving via the Finite Element method the forward problem \eqref{prob:cavity}, with boundary conditions prescribed as in Figure \ref{fig:bound_cond} on the square, with one or more cavities of given geometries. We use a tessellation $\mathcal{T}^{ref}_h$ which is more refined than $\mathcal{T}_h$ on the common part outside the cavities (see Figure \ref{fig:meshes} for an example of the two tessellations) in order not to commit inverse crime. Once extracting the values of the solution of the forward problem on the boundary of the domain $\Omega$ obtained by the mesh $\mathcal{T}^{ref}_h$, we interpolate these values on the mesh $\mathcal{T}_h$.  
\begin{figure}[h!]
    \centering
    \begin{subfigure}{0.4\textwidth}
    \centering
    \includegraphics[width=\textwidth]{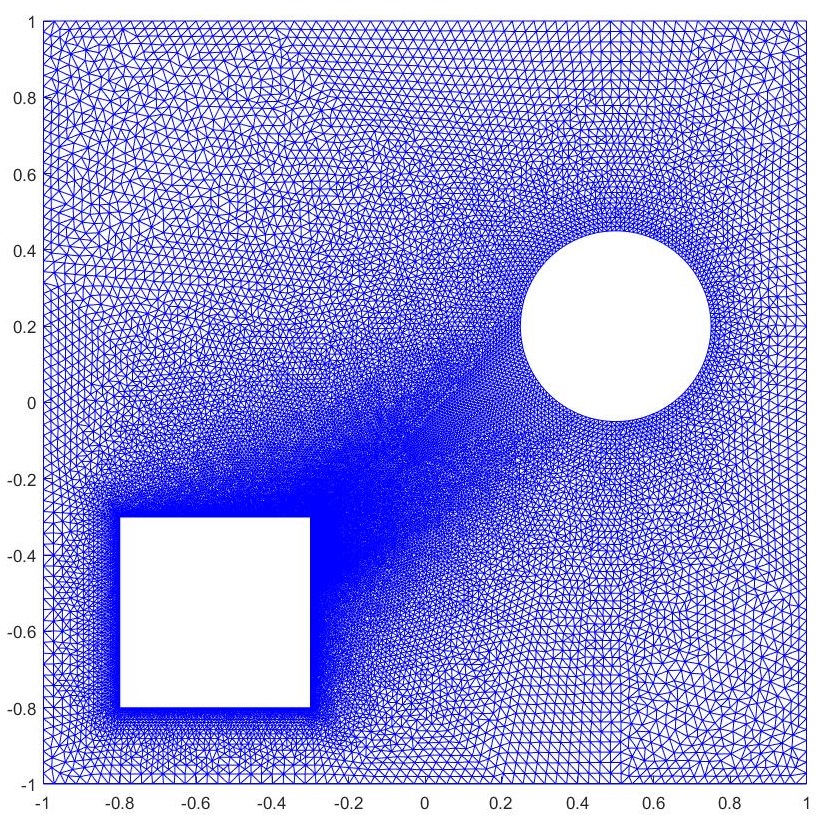}
    \caption{Mesh $\mathcal{T}^{ref}_h$: forward problem.}
    \label{fig:mesh_cavity}
    \end{subfigure}
    \hfill
    \begin{subfigure}{0.4\textwidth}
    \centering
    \includegraphics[width=\textwidth]{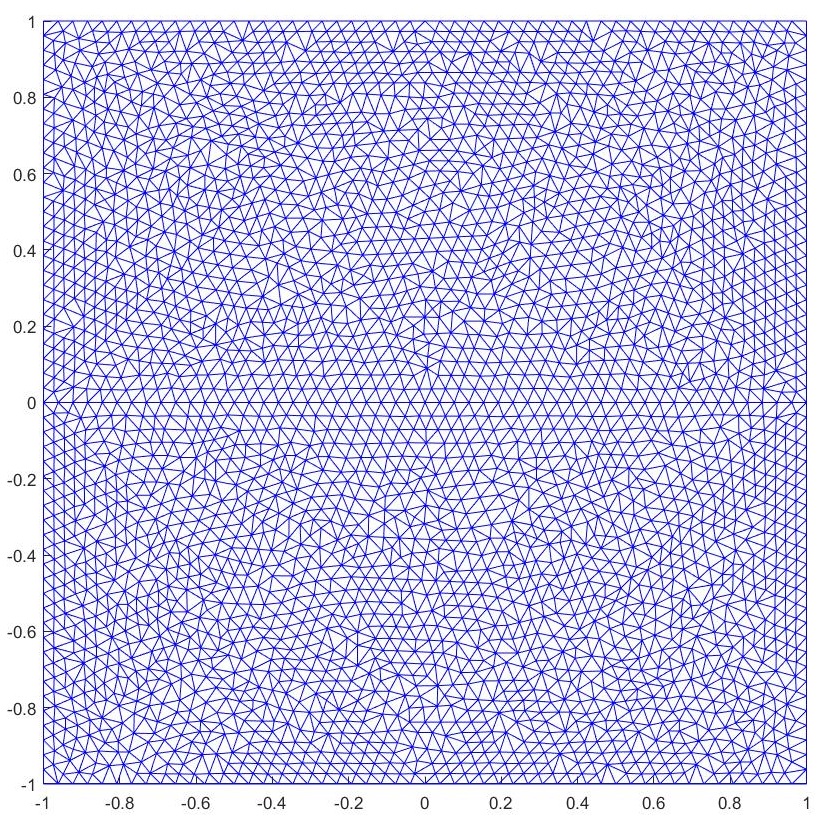}
    \caption{Mesh $\mathcal{T}_h$: inverse problem.}
    \label{fig:mesh_inverse}
    \end{subfigure}
\caption{Example of the meshes adopted.}
\label{fig:meshes}    
\end{figure}
Therefore, by $u_{meas}$ we denote the resulting boundary datum on the mesh $\mathcal{T}_h$. 
We also mention that the triangular mesh is adaptively refined during the reconstruction procedure using the values of $\nabla v_h$ after an a-priori fixed number of iterations which depend on the specific numerical example. See, as example, Figure \ref{fig:adapt_mesh} related to the reconstruction of a circular cavity.
\begin{figure}[h!]
    \centering
    \begin{subfigure}{0.4\textwidth}
    \includegraphics[width=\textwidth]{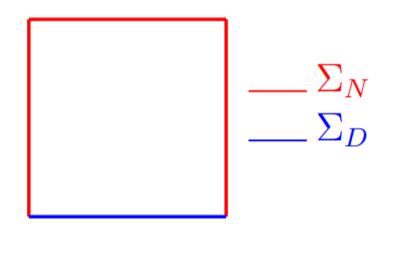}
    \caption{Boundary condition in numerical experiments: Neumann boundary conditions are assigned on the red part. Homogeneous Dirichlet conditions are assigned on the blue part. }
    \label{fig:bound_cond}
    \end{subfigure}
    \hfill
    \begin{subfigure}{0.4\textwidth}
    \centering
    \includegraphics[width=\textwidth]{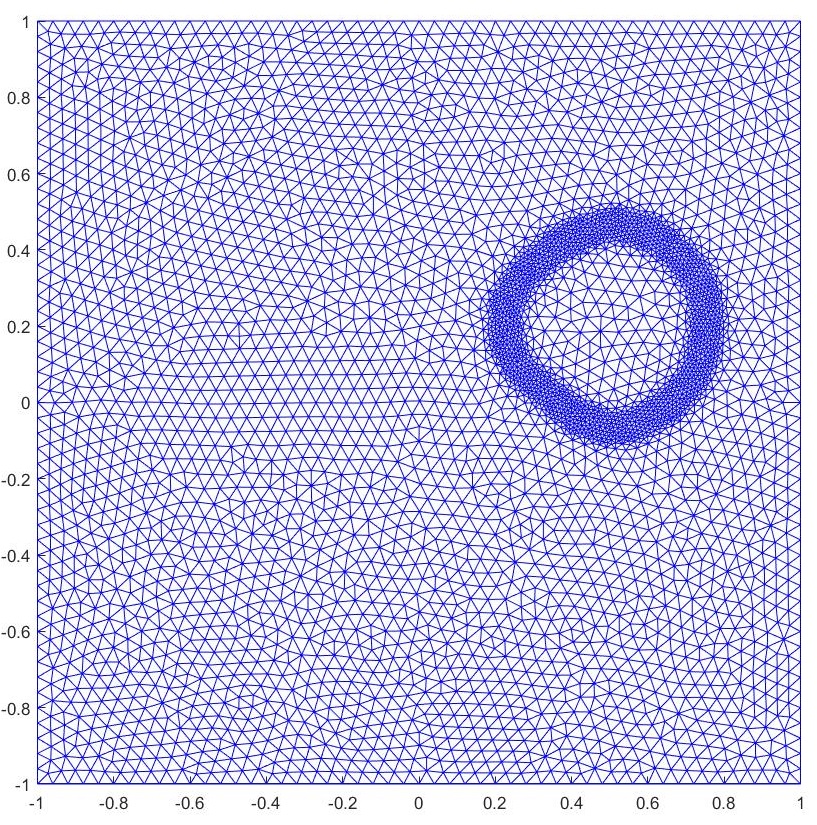}
    \caption{Refinement of the mesh around the reconstructed domain. This is the mesh at the final iteration of the experiment in Figure \ref{fig:test2b}.}
    \label{fig:adapt_mesh}
    \end{subfigure}
    \caption{Geometrical setting and refinement of the mesh.}
    \label{fig:geom_square}
\end{figure}

In the reconstruction procedure, i.e. for the implementation of the Algorithm \ref{al:algorithm}, we assume to know two different boundary measurements. In fact, in the context of inverse boundary value problems of this kind, it is reasonable to use $N_g>1$ different boundary measurements $u^{i}_{meas}$, for $i=1,\ldots,N_g$ which clearly improve the numerical reconstruction results. Thus, we consider a slight modification of the original optimization problem \eqref{minrel}, assuming the knowledge of $N_g$ different Neumann boundary data $g^i$, for $i=1,\ldots,N_g$ and hence considering 
\begin{equation}\label{eq:min_prob_more_meas}
\begin{aligned}
\min_{v \in \mathcal{K}} J^{sum}_{\delta,\varepsilon}(v),& \,\,\, \\ &\hspace{-2.1cm} J^{sum}_{\delta,\varepsilon}(v) := \frac{1}{N_g}\sum_{i=1}^{N_g}\left( \frac12 \|u^{i}_{\delta}(v)-u^{i}_{meas}\|_{L^2(\Sigma_N)}^2\right)
+ \widetilde{\alpha} \!\int_{\Omega}\Big( \varepsilon|\nabla v|^2 + \frac{1}{\varepsilon}v(1-v)\Big),
\end{aligned}
\end{equation}
where $u^{i}_{\delta}(v)\in H^1_{\Sigma_D}(\Omega)$ is the solution to \eqref{elasticity} with $g=g^{i}$ and for $v\in \mathcal{K}$. The necessary optimality condition related to  \eqref{eq:min_prob_more_meas} can be equivalently obtained reasoning similarly as we did to derive \eqref{VI}.

In Table \ref{tab}, we collect some of the parameters utilized in most numerical tests. Possible changes in these values are highlighted in the text related to each specific experiment. 

Finally, all the numerical experiments are performed choosing, as initial guess, the phase-field variable $v_0=0$.

\begin{table}[h!]
{\begin{center}
 \begin{tabular}{|c | c | c | c | c|} 
 \hline
 \centering
 tol & $\widetilde{\alpha}$ & $\tau_n$ & $\varepsilon$ & $\delta$  \\ [0.5ex] 
 \hline\hline
 $10^{-5}$ & $10^{-2}$ & $10^{-3}$ & $\frac{1}{16\pi}$ or $\frac{1}{8\pi}$  & $10^{-2}$ \\ [1ex]
 \hline
 \end{tabular}
 \captionof{table}{Values of some parameters utilized in Algorithm \ref{al:algorithm}.}\label{tab}
\end{center}
}
\end{table}

\subsection{Numerical experiments with $N_g=2$ and without noise in the measurements.}

\textbf{Test 1: reconstruction of a circular cavity.} The elastic medium is described by the Lam\'e parameters $\mu=0.2$ and $\lambda=1$. The Neumann boundary conditions are $g^1(x,y)=(0, \frac{1}{10}-\frac{3}{10}y)$ and $g^2(x,y)=(-\frac{1}{2}x^2, y^2)$. We set the parameter $\varepsilon=\frac{1}{16\pi}$. 
The mesh is refined with respect to the gradient of the phase-field variable every $1000$ iterations. The algorithm stops after $n=3544$ iterations. In Figure \ref{fig:Test1} we show the numerical results at three different time steps. 
\begin{figure}[h!]
    \centering
    \begin{subfigure}{0.4\textwidth}
    \centering
    \includegraphics[width=\textwidth]{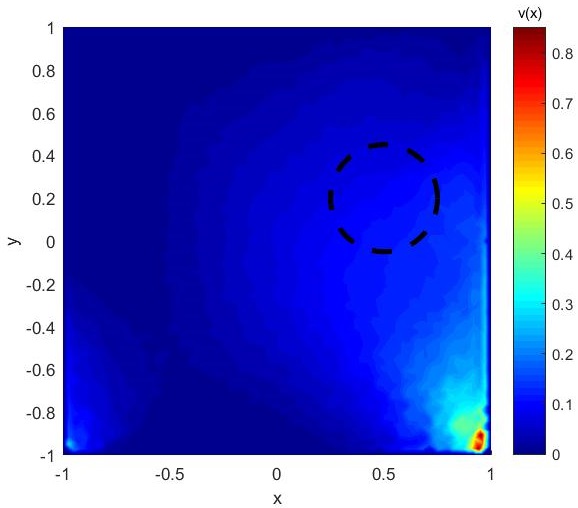}
    \caption{At $n=20$}
    \label{fig:test1}
    \end{subfigure}
    \centering
    \begin{subfigure}{0.4\textwidth}
    \centering
    \includegraphics[width=\textwidth]{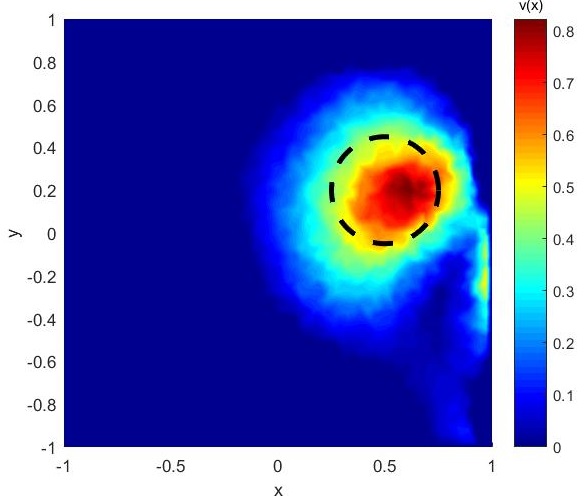}
    \caption{At $n=200$}
    \label{fig:test1_200}
    \end{subfigure}
    \hfill \\
    \begin{subfigure}{0.3\textwidth}
    \centering
    \includegraphics[width=\textwidth]{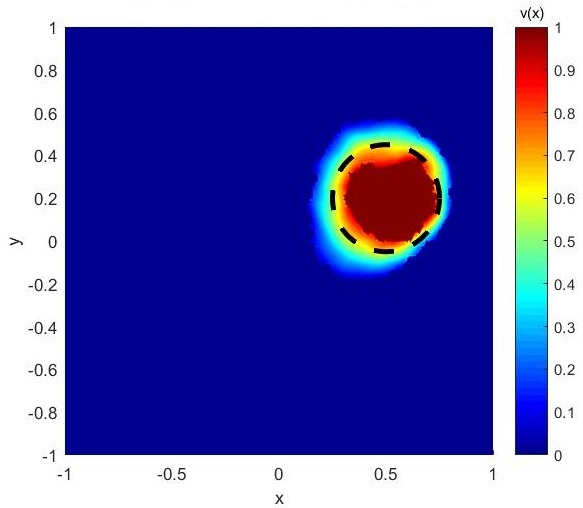}
    \caption{At $n=1000$}
    \label{fig:test1a}
    \end{subfigure}
    \hfill
    \begin{subfigure}{0.3\textwidth}
    \centering
    \includegraphics[width=\textwidth]{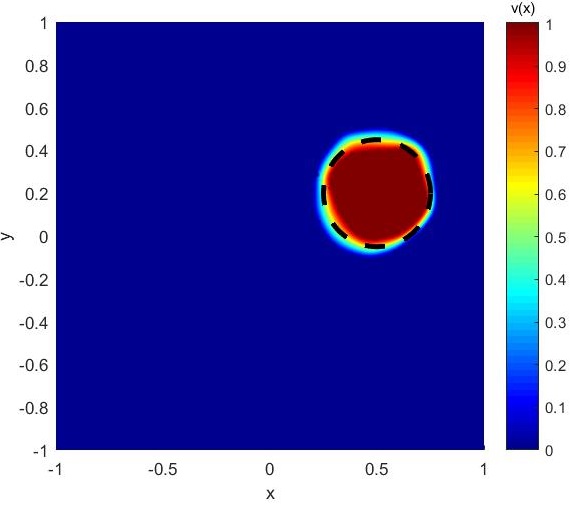}
    \caption{At $n=2000$}
    \label{fig:test1b}
    \end{subfigure}
    \hfill
    \begin{subfigure}{0.3\textwidth}
    \centering
    \includegraphics[width=\textwidth]{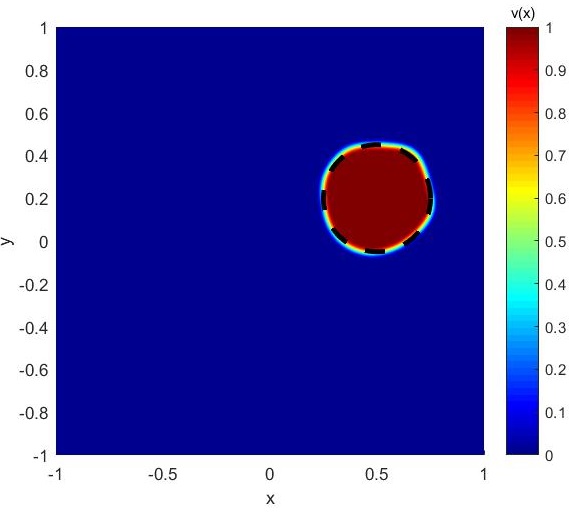}
    \caption{$n=3544$, final step}
    \label{fig:test1c}
    \end{subfigure}
\caption{Test 1. Reconstruction of a circular cavity. Dotted line represents the target cavity.}
\label{fig:Test1}    
\end{figure}

\textbf{Test 2: reconstruction of a circular cavity - changing boundary conditions and Lam\'e parameters.} We propose the same numerical experiments of Test 1, showing how the results change using different Neumann boundary conditions and Lam\'e parameters. We report in the captions of Figure \ref{fig:Test2} the selected parameters, data, and also the number of time steps needed for reaching the tolerance. Note that the three experiments consider different values for the Poisson coefficient $\nu:=\frac{\lambda}{2(\lambda+\mu)}$, that is $\nu=\frac{1}{4}$, $\nu=\frac{1}{3}$, and $\nu=-\frac{1}{18}$, respectively. In the three numerical examples of Figure \ref{fig:Test2}, the refinement of the mesh happens every $1500$, $1000$, $2000$ iterations, respectively. 

\begin{figure}[h!]
    \centering
    \begin{subfigure}{0.3\textwidth}
    \centering
    \includegraphics[width=\textwidth]{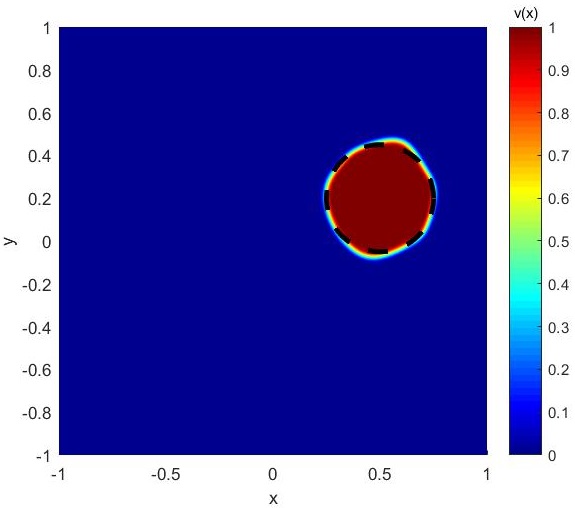}
    \caption{$ (\mu,\lambda)=(1,1);\vspace{0.1cm} \\ g^1(x,y)=(2,0);\vspace{0.1cm} \\ g^2(x,y)=(-\frac{1}{2}x^2,y^2);\vspace{0.1cm} \\ n=4308.$}
    \label{fig:test2a}
    \end{subfigure}
    \hfill
    \begin{subfigure}{0.3\textwidth}
    \centering
    \includegraphics[width=\textwidth]{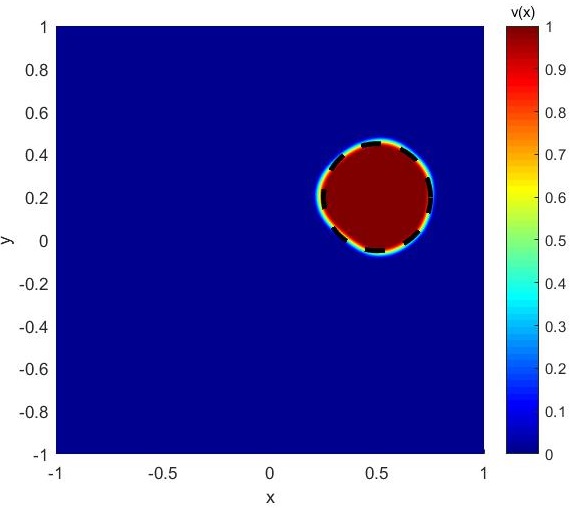}
    \caption{
    $ (\mu,\lambda)=(0.5,1);\vspace{0.1cm} \\ g^1(x,y)=(x,y);\vspace{0.1cm} \\ g^2(x,y)=(-y,-x);\vspace{0.1cm} \\ n=5375.$}
    \label{fig:test2b}
    \end{subfigure}
    \hfill
    \begin{subfigure}{0.3\textwidth}
    \centering
    \includegraphics[width=\textwidth]{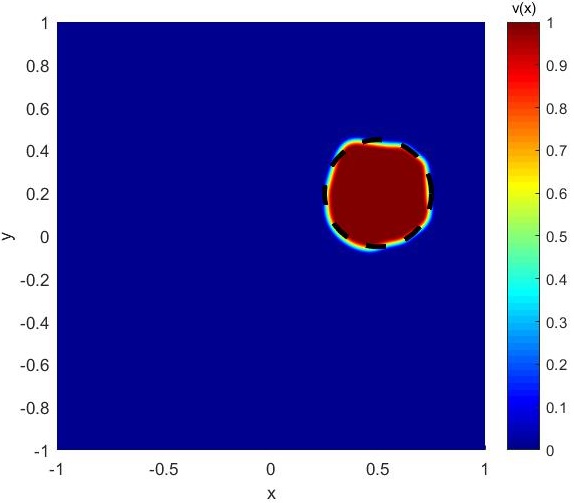}
    \caption{$ (\mu,\lambda)=(2,-0.2);\vspace{0.1cm} \\  g^1(x,y)=(5x,4y);\vspace{0.1cm} \\ g^2(x,y)=(-3y,-3x);\vspace{0.1cm} \\ n=3362.$}
    \label{fig:test2c}
    \end{subfigure}
\caption{Test 2. Reconstruction of a circular cavity using several parameters and data. For each experiment, we report the configuration at the final step $n$. Dotted line represents the target cavity.}
\label{fig:Test2}    
\end{figure}

\textbf{Test 3: reconstruction of a Lipschitz domain.} This experiment aims at reconstructing a square-shaped cavity. We show several numerical tests, choosing different values for $\varepsilon$, different boundary conditions and different values of the number of iterations for the refinement of the mesh. We have already shown results based on different choices for the values of the Lam\'e parameters in the previous numerical tests, so we fix the values of Lam\'e coefficients to be $\mu=0.5$ and $\lambda=1$. In fact, recalling that the range of the Poisson coefficient is $-1<\nu<\frac{1}{2}$ ($\nu=\frac{1}{2}$ represents the incompressible case), we have considered four relevant cases for the Poisson coefficient: one test on an elastic material close to incompressible case ($\nu=\frac{5}{12}$ in Figure \ref{fig:test1c}), two tests on elastic coefficients of common materials ($\nu=\frac{1}{4}$ and $\nu=\frac{1}{3}$ in Figures \ref{fig:test2a} and \ref{fig:test2b}, respectively), and one test on auxetic materials, that is materials with negative Poisson ratio ($\nu=-\frac{1}{18}$ in Figure \ref{fig:test2c}). In the results of Figure \ref{fig:Test3}, the refinement of the mesh happens every 6000 for the first two experiments and every 3000 iterations for the last one. The second numerical result, see Figure \ref{fig:test3b}, has the same parameters of the numerical example of Figure \ref{fig:test3a} except $\widetilde{\alpha}$ which is chosen $\widetilde{\alpha}=5\times 10^{-2}$.

\begin{figure}[h!]
    \centering
    \begin{subfigure}{0.3\textwidth}
    \centering
    \includegraphics[width=\textwidth]{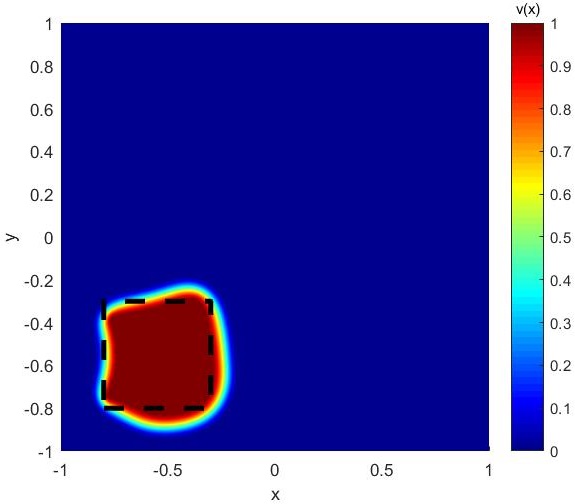}
    \caption{$ \varepsilon=\frac{1}{8\pi};\vspace{0.1cm} \\ g^1(x,y)=(\frac{1}{10},0);\vspace{0.1cm} \\ g^2(x,y)=(-\frac{1}{2}x^2,y^2);\vspace{0.1cm} \\ n=7957.$}
    \label{fig:test3a}
    \end{subfigure}
    \hfill
    \begin{subfigure}{0.3\textwidth}
    \centering
    \includegraphics[width=\textwidth]{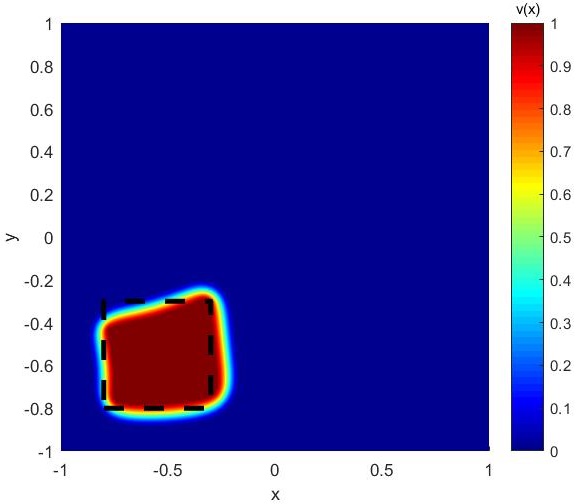}
    \caption{$ \varepsilon=\frac{1}{8\pi}; \widetilde{\alpha}=5\times10^{-2};\vspace{0.1cm} \\ g^1(x,y)=(0.1,0);\vspace{0.1cm} \\ g^2(x,y)=(-0.5x^2,y^2);\vspace{0.1cm} \\ n=16346.$}
    \label{fig:test3b}
    \end{subfigure}
    \hfill
    \begin{subfigure}{0.3\textwidth}
    \centering
    \includegraphics[width=\textwidth]{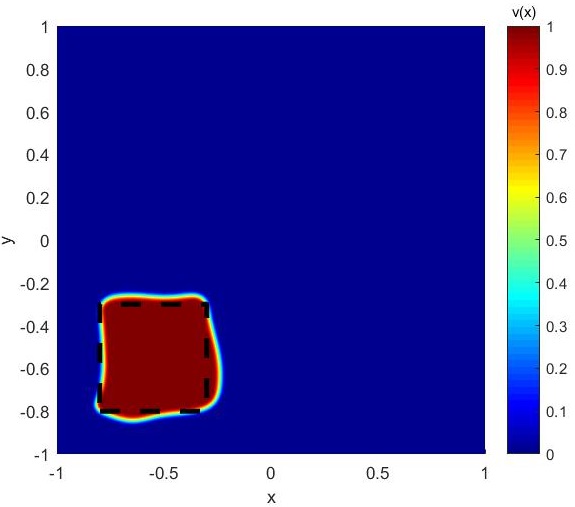}
    \caption{$ \varepsilon=\frac{1}{16\pi}; \vspace{0.1cm}\\  g^1(x,y)=(0,\frac{2}{5}x-\frac{3}{10}y);\vspace{0.1cm} \\ g^2(x,y)=(-\frac{1}{2}x^2,y^2);\vspace{0.1cm} \\ n=10931.$}
    \label{fig:test3c}
    \end{subfigure}
\caption{Test 3. Reconstruction of a square-shaped cavity. Dotted line represents the target cavity.}
\label{fig:Test3}    
\end{figure}

\textbf{Test 4: reconstruction of two cavities.} This test provides results when the two cavities to be reconstructed are a square and a circle. Neumann boundary conditions are given by $g^1(x,y)=(x,y)$ and $g^2(x,y)=(-y,-x)$. We propose two numerical reconstruction procedures, see Figure \ref{fig:Test4}. In Figure \ref{fig:test4a}, we report the results obtained by the standard algorithm, while in Figure \ref{fig:test4c} we use a variant of the Algorithm \ref{al:algorithm} where the parameter $\varepsilon$ is initially set $\varepsilon=\frac{1}{4\pi}$ but after a fixed and a-priori chosen number of iterations (8000 iterations) is updated and set $\varepsilon=\varepsilon/4$. In both cases the mesh is refined after 5000 iterations. It is worth noting that the variant of Algorithm \ref{al:algorithm} does not produce the visible oscillations of the test in Figure \ref{fig:test4a}.\\ 
Note that we also change a little bit the value of $\delta$. We have observed that $\delta$ cannot be chosen too small otherwise numerical instability can appear. Numerically we have seen that, in order to overcome this issue, $\tau_n$ has to be chosen always smaller than $\delta$. However, choosing $\tau$ too small increases the number of necessary iterations to satisfy the stopping criterium.    
\begin{figure}[h!]
    \centering
    \begin{subfigure}{0.4\textwidth}
    \centering
    \includegraphics[width=\textwidth]{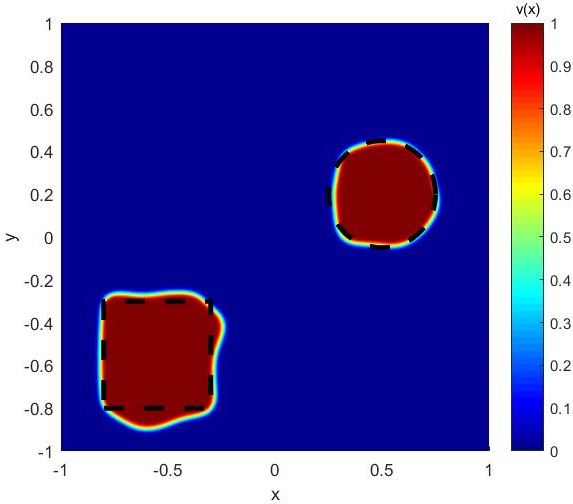}
    \caption{$ \varepsilon=\frac{1}{16\pi};\  \delta=10^{-2}; n=8531$.\ \\ \ \\ \ }
    \label{fig:test4a}
    \end{subfigure}
    \hfill
    \begin{subfigure}{0.4\textwidth}
    \centering
    \includegraphics[width=\textwidth]{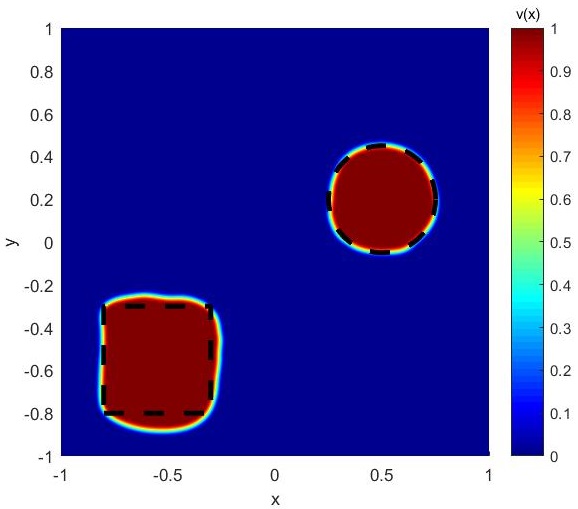}
    \caption{
    $ \varepsilon=\frac{1}{4\pi}$ for $n\leq 7999$;\vspace{0.1cm}\\ $\varepsilon=\frac{1}{16\pi}$ for $n\geq 8000$;\vspace{0.1cm}\\ 
    $\  \delta=7.5\times 10^{-2};\ n=10852.$}
    \label{fig:test4c}
    \end{subfigure}
\caption{Test 4. Reconstruction of two cavities. Dotted line represents the target cavity.}
\label{fig:Test4}    
\end{figure}

\textbf{Test 5: reconstruction of a non-convex domain.} We finally propose the reconstruction of a cavity which is not convex, see Figure \ref{fig:Test5}. We use $g^1(x,y)=(x,y)$ and $g^2(x,y)=(-y,-x)$ as Neumann boundary conditions and $\mu=0.5$ and $\lambda=1$. Parameters have the following values: $\varepsilon=\frac{1}{16\pi}$, and $\tau_n=5\times 10^{-4}$. Mesh is refined every 5000 iterations. The stopping criterium is satisfied after $n=6825$ iterations.
\begin{figure}[h!]
    \centering
    \includegraphics[width=0.5\textwidth]{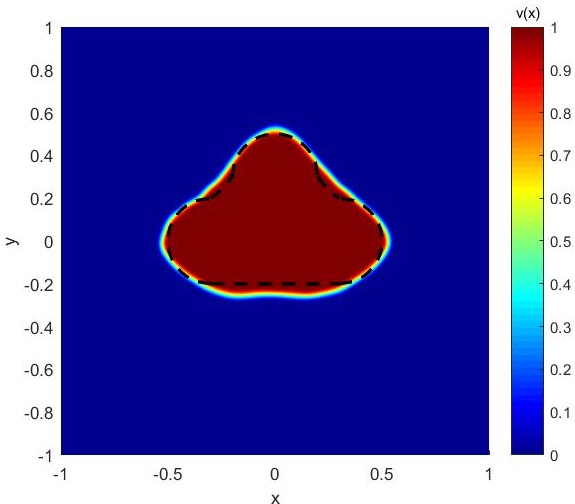}
\caption{Test 5. Reconstruction of a non-convex domain. It seems that the algorithm tends to reconstruct a convex domain. Dotted line represents the target cavity.}
\label{fig:Test5}    
\end{figure}

\subsection{Numerical experiments with $N_g=2$ and noise in the measurements.}
\textbf{Test 6: reconstruction of cavities of different shapes using noisy measurements.} Here  we run some of the numerical tests showed in the previous section, adding to the boundary measurements a normal distributed noise with zero mean and variance equal to one. We choose two different noise levels: $2\%$ and $5\%$. The results are reported in Figure \ref{fig:Test6}. \\
For the the test in Figure \ref{fig:test6a} and Figure \ref{fig:test6b}, we use values of parameters as in Test 1 and refine the mesh every $2000$ and $2500$ iterations, respectively. The reconstruction of a square-shaped cavity, that is Figure \ref{fig:test6c} and Figure \ref{fig:test6d}, are obtained by means of parameters of Test 3 - Figure \ref{fig:test3c}, refining the mesh every $3000$ and $10000$ iterations. Lastly, to get the results in Figure \ref{fig:test6e} and Figure \ref{fig:test6f} we use the same parameters of Test 4 - Figure \ref{fig:test4c}. The mesh is refined every $5000$ and $8000$ iterations, while the value of the parameter $\varepsilon$ is adapted after $8000$ and $10000$ iterations, respectively.    
In the captions of the single figures, we specify the values that are changed with respect to the ones proposed in the Tests 1, 3, and 4. 
\begin{figure}
    \centering
    \begin{subfigure}{0.4\textwidth}
    \centering
    \includegraphics[width=\textwidth]{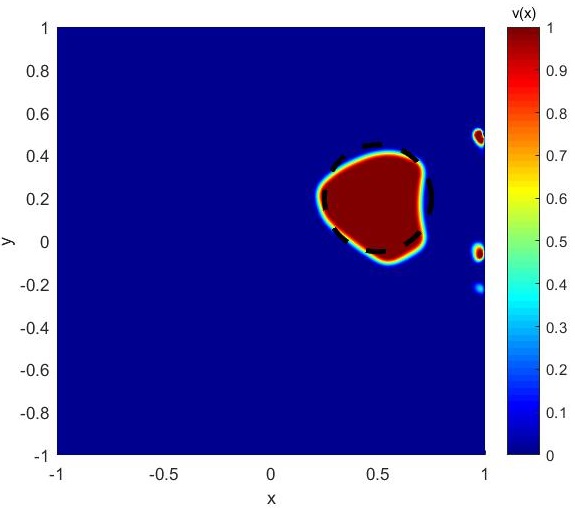}
    \caption{$(\mu,\lambda)=(0.2,1)$, $\tau_n=5\times 10^{-4}$.\vspace{0.1cm}\\ $n=11071$.  Noise level $2\%$.}
    \label{fig:test6a}
    \end{subfigure}
    \hfill
    \begin{subfigure}{0.4\textwidth}
    \centering
    \includegraphics[width=\textwidth]{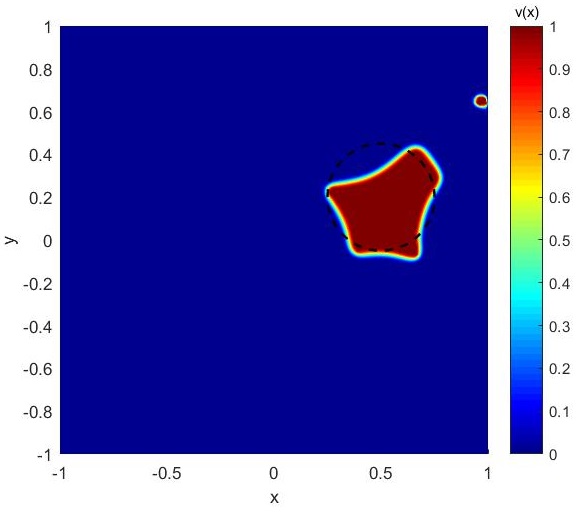}
    \caption{$(\mu,\lambda)=(0.2,1)$, $\tau_n=5\times 10^{-4}$.\vspace{0.1cm}\\ $\widetilde{\alpha}=5\times 10^{-2}$, $n=14361$.\vspace{0.1cm} \\ Noise level $5\%$.}
    \label{fig:test6b}
    \end{subfigure}
    \hfill
    \begin{subfigure}{0.4\textwidth}
    \centering
    \includegraphics[width=\textwidth]{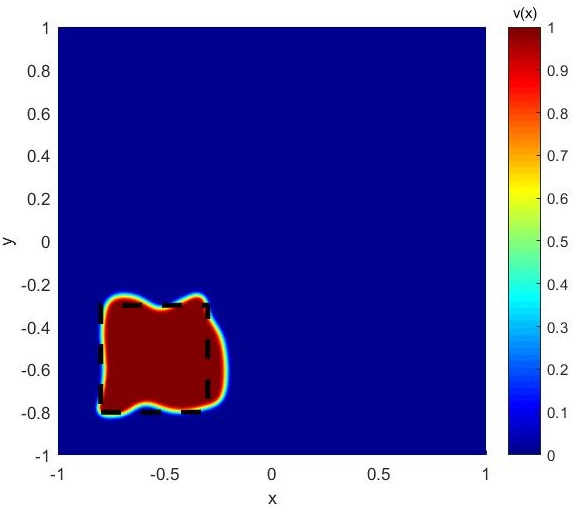}
    \caption{$(\mu,\lambda)=(0.5,1)$, $\tau_n=5\times 10^{-4}$.\vspace{0.1cm}\\ $n=23652$. 
    Noise level $2\%$.}
    \label{fig:test6c}
    \end{subfigure}
    \hfill
    \begin{subfigure}{0.4\textwidth}
    \centering
    \includegraphics[width=\textwidth]{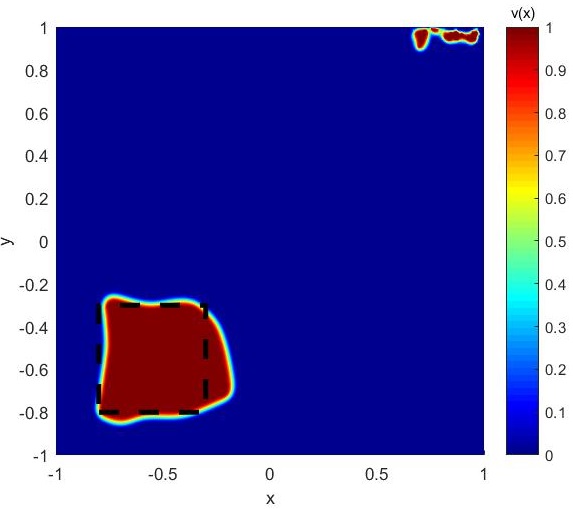}
    \caption{$(\mu,\lambda)=(0.5,1)$, $\tau_n=5\times 10^{-4}$.\vspace{0.1cm}\\ $n=15854$. 
    Noise level $5\%$.}
    \label{fig:test6d}
    \end{subfigure}
    \hfill
    \begin{subfigure}{0.4\textwidth}
    \centering
    \includegraphics[width=\textwidth]{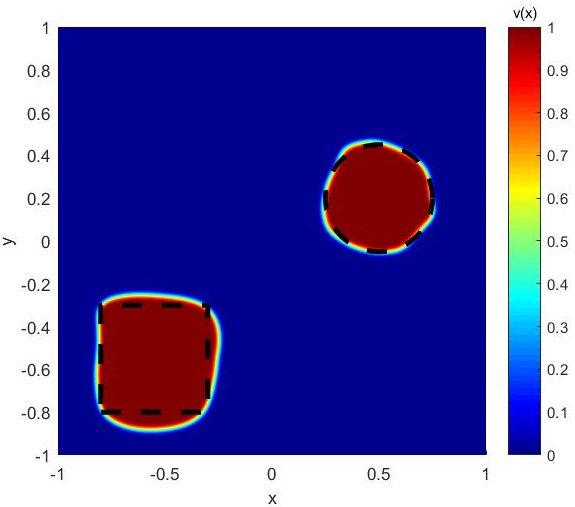}
    \caption{$(\mu,\lambda)=(0.5,1)$. \vspace{0.1cm}\\ $n=11776$. Noise level $2\%$.}
    \label{fig:test6e}
    \end{subfigure}
    \hfill
    \begin{subfigure}{0.4\textwidth}
    \centering
    \includegraphics[width=\textwidth]{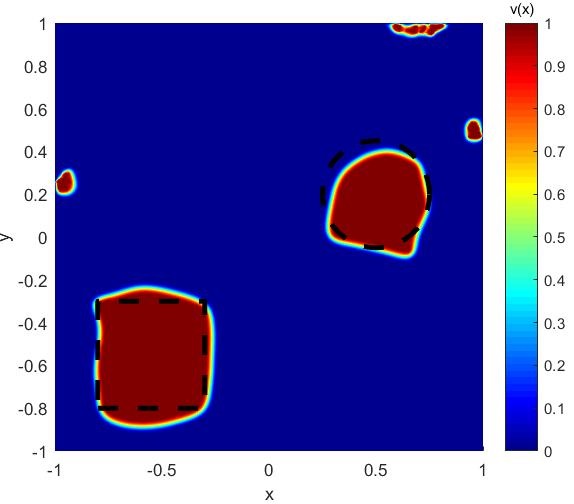}
    \caption{$(\mu,\lambda)=(0.5,1)$, $\tau_n=5\times 10^{-4}$.\vspace{0.1cm}\\ $n=25480$. Noise level $5\%$.}
    \label{fig:test6f}
    \end{subfigure}
\caption{Test 6. Reconstruction of cavities by means of noisy measurements. Dotted line represents the target cavity.}
\label{fig:Test6}    
\end{figure}

\newpage

\section*{Acknowledgments}
The authors deeply thank Dorin Bucur and Alessandro Giacomini for suggesting relevant literature and for useful discussions that led us to improve some of the results in this work. 

This research has been partially performed in the framework of the MIUR-PRIN Grant 2020F3NCPX ``Mathematics for industry 4.0 (Math4I4)''.

Andrea Aspri, Cecilia Cavaterra and Elisabetta Rocca are members of GNAMPA (Gruppo Nazionale per l’Analisi Matematica, la Probabilità e le loro Applicazioni) of INdAM (Istituto Nazionale
di Alta Matematica).

Marco Verani has been partially funded by MIUR PRIN research grants n. 201744KLJL and n. 20204LN5N5. Marco Verani is a member of GNCS (Gruppo Nazionale per il Calcolo Scientifico) of INdAM.

\bibliographystyle{plain}
\bibliography{references}
%\printbibliography

\end{document}